\documentclass[11pt]{amsart}
\usepackage[foot]{amsaddr}
\usepackage{ourstyle}
\usepackage{libertine}
\usepackage[T1]{fontenc}
\usepackage{euler}

\author{Sushovan Majhi}
\address{University of California, Berkeley, USA}
\email{smajhi@berkeley.edu}

\title{Vietoris--Rips Complexes of Metric Spaces Near a Metric Graph}

\begin{document}

\begin{abstract}
For a sufficiently small scale $\beta>0$, the Vietoris--Rips complex
$\mathcal{R}_\beta(S)$ of a metric space $S$ with a small Gromov--Hausdorff
distance to a closed Riemannian manifold $M$ has been already known to recover
$M$ up to homotopy type. While the qualitative result is remarkable and
generalizes naturally to the recovery of spaces beyond Riemannian
manifolds---such as geodesic metric spaces with a positive convexity
radius---the generality comes at a cost. Although the scale parameter $\beta$ is
known to depend only on the geometric properties of the geodesic space, how to
quantitatively choose such a $\beta$ for a given geodesic space is still
elusive. In this work, we focus on the topological recovery of a special type of
geodesic space, called a metric graph. For an abstract metric graph
$\mathcal{G}$ and a (sample) metric space $S$ with a small Gromov--Hausdorff
distance to it, we provide a description of $\beta$ based on the convexity
radius of $\mathcal{G}$ in order for $\mathcal{R}_\beta(S)$ to be homotopy
equivalent to $\mathcal{G}$. Our investigation also extends to the study of the
Vietoris--Rips complexes of a Euclidean subset $S\subset\mathbb{R}^d$ with a
small Hausdorff distance to an embedded metric graph
$\mathcal{G}\subset\mathbb{R}^d$. From the pairwise Euclidean distances of
points of $S$, we introduce a family (parametrized by $\varepsilon$) of
path-based Vietoris--Rips complexes $\mathcal{R}^\varepsilon_\beta(S)$ for a
scale $\beta>0$. Based on the convexity radius and distortion of the embedding
of $\mathcal{G}$, we show how to choose a suitable parameter $\varepsilon$ and a
scale $\beta$ such that $\mathcal{R}^\varepsilon_\beta(S)$ is homotopy
equivalent to $\mathcal{G}$.
\end{abstract}

\maketitle

\section{Introduction}
In this paper, we study the homotopy type of the Vietoris--Rips complexes of a
metric space near a metric graph (\defref{metric-graph}). Given a metric space
$(X,d_X)$ and a positive scale $\beta$, the Vietoris--Rips complex
$\Ri_\beta(X)$ is defined as an abstract simplicial complex having a
$k$--simplex for every finite subset of $X$ with cardinality $(k+1)$ and
diameter (\defref{diam}) less than $\beta$. 

The construction of the Vietoris--Rips complex $\Ri_\beta(X)$ can be loosely
thought of as the ``fattening'' of the metric space $X$ by an amount $\beta$.
For a well-behaved metric space $X$ and sufficiently small scale $\beta$, the
``thickened'' $X$ is expected to be homotopy equivalent to $X$. For sufficiently
small $\beta$, Hausmann showed in \cite{hausmann_1995} that a closed Riemannian
manifold $M$ is homotopy equivalent to its Vietoris--Rips complex
$\Ri_\beta(M)$. Latschev further guarantees in \cite{latschev_2001} the
existence of a sufficiently small scale $\beta$ such that $M$ is homotopy
equivalent to the Vietoris--Rips complex $\Ri_\beta(S)$ of a Gromov--Hausdorff
close metric space $(S,d_S)$. Despite the qualitative guarantee in
\cite{latschev_2001}, it is not apparently clear how to quantitatively choose
such a small $\beta$ for a given Riemannian manifold. The choice of a suitable
$\beta$ is still elusive---it requires explicit knowledge of the curvature
bounds of $M$. It is reasonable to ask if we can quantitatively choose $\beta$
for some special cases.

Although the works of Hausmann (\cite{hausmann_1995}) and Latschev
(\cite{latschev_2001}) consider $M$ to be a closed Riemannian manifold, the
results hold true more generally for a geodesic metric space with a positive
convexity radius; see the remark \cite[p.~179]{hausmann_1995}. Roughly speaking,
the convexity radius of a geodesic space is the radius of the largest
(geodesically) convex ball. A metric graph, denoted $\G$ throughout the paper,
is a special type of geodesic space; see \defref{metric-graph}. In general, the
study of the topology of the Vietoris--Rips complexes of a metric space $S$ is
very delicate. The intrinsic filamentary structure and a positive convexity
radius of $\G$, however, facilitate such a pursuit when $S$ is assumed to be in
a close proximity to $\G$. Our current study revolves around the recovery (up to
homotopy type) of a metric graph $(\G,d_\G)$ from the Vietoris--Rips complexes
of a (sample) metric space $(S,d_S)$ with a small Gromov--Hausdorff and
Hausdorff distance to $\G$.

\subsection{Motivation}
In the last decade, the Vietoris--Rips complexes have received an increasing
popularity in the computational topology and topological data analysis (TDA)
community. In the shape reconstruction paradigm, for example, a point-cloud can
be modeled as a finite metric space $(S,d_S)$ sampled around an unknown
\emph{shape} $(X,d_X)$. The objective then is to \emph{learn} the topology of
$X$ by studying different simplicial complexes of $S$. Some of the popular
choices are Vietoris--Rips complex, \v{C}ech complex, witness complex
\cite{DeSilva2003AWD}, alpha complex \cite{Edelsbrunner2009AlphaS}, etc. In many
applications, the Vietoris--Rips complexes are deemed a better alternative to
the conventional \v{C}ech complexes. The computational scheme of \v{C}ech
complexes is not well-understood and requires explicit knowledge of the ambient
metric space, while the Vietoris--Rips complexes can be computed just from the
pairwise distances of points of $S$. 

The topological reconstruction of an embedded smooth manifold---more generally
spaces with a positive reach---using the Vietoris--Rips complexes has been
studied, for example, in
\cite{Adams_2019,ATTALI2013448,co-tpbr-08,kim2020homotopy}. Such results
embolden us to employ the Vietoris--Rips complexes in the recovery of spaces
beyond the class of positive reach. For an example of such a space, consider a
graph embedded in the Euclidean space. There is a wide spectrum of data-driven
applications requiring the reconstruction of a (hidden) metric graph, both
abstract and embedded, from a finite sample around it. Examples include road-map
reconstruction from GPS traces
\cite{Ahmed:2015:CTD:2820783.2820810,dey_graph_2018_socg}, reconstruction of the
filamentary trajectory of shock from earthquake sensors
\cite{geological-survey}, etc. Among the notable Vietoris--Rips inspired
reconstruction of metric graphs, we mention \cite{aanjaneya2012metric} and
\cite{fasy2018reconstruction,Lecci:2014:SAM:2627435.2697074} under the
Gromov--Hausdorff and Hausdorff sampling conditions, respectively.

\subsection{Related Work}
The current work is discerned to be closely related to and inspired by the
reconstruction of embedded geodesic spaces from a finite sample
\cite{fasy2018reconstruction}. For a Hausdorff--close sample $S$ around an
embedded metric graph $\G$, the authors note that the Euclidean Vietoris--Rips
complex generally fails to be homotopy equivalent to the underlying graph. It
does not come as a surprise---because a metric graph can have very sharp
corners, moreover a small Hausdorff distance does not guarantee a small
Gromov--Hausdorff distance between the sample and the ground truth. The
Euclidean Vietoris--Rips complex does not serve well as a topologically faithful
\emph{reconstruction} of the unknown graph. Instead of the Euclidean metric, the
Vietoris--Rips complexes of the sample under a family of path-based metrics
$(S,d^\eps)$ (defined in \defref{d-eps}) is introduced. Under this metric, the
authors show that the Vietoris--Rips complex $\Ri^\eps_\beta(S)$ and the
underlying metric graph $\G$ have isomorphic fundamental groups; see
\cite[Theorem 4.3]{fasy2018reconstruction}. We further investigate this metric
to find that the result holds for higher homotopy groups as well, as previously
conjectured in \cite{fasy2018reconstruction}. In this context, we also mention
the works of \cite{Adamaszek2018,Adams_2019}, where the homotopy equivalence
results of Hausmann \cite{hausmann_1995} and Latschev \cite{latschev_2001} have
been recast in the light of an alternative metric, called the Vietoris--Rips
thickening, in order to additionally retain the metric structure.

\subsection{Our Contribution}
One of the major contributions of this work is to quantify the scale parameter
at which the Vietoris--Rips complex of a sample $S$ recovers (up to homotopy
type) a metric graph, under both the Gromov--Hausdorff and Hausdorff sampling
conditions. Our main homotopy equivalence results are presented in
\thmref{gh-hom} and \thmref{h-hom}, respectively. The work claims novelty in
using the barycentric subdivision (defined in \secref{prelim}) as a critical
ingredient in establishing the homotopy equivalences.

This paper is organized in the following manner. \secref{prelim} contains
definitions, notations, and facts that are frequently used throughout the paper.
In \secref{abs-gr}, the recovery of an abstract metric graph $\G$ from a
Gromov--Hausdorff close sample is obtained. When the convexity radius $\rho(\G)$
is positive, \thmref{gh-hom} proves the homotopy equivalence between $\G$ and
the Vietoris--Rips complex of the sample for a sufficiently small, positive scale
$\beta$.
\begin{restatable*}[Homotopy Equivalence under Gromov-Hausdorff Distance]
	{theorem}{ghhom}\label{thm:gh-hom}
Let $(\G,d_\G)$ be a compact, path-connected metric graph, $(S,d_S)$ a metric
space, and $\beta>0$ a number such that 
	\[
		3d_{GH}(\G,S)<\beta<\frac{3\rho(\G)}{4}.
	\] Then, $\mod{\Ri_\beta(S)}\simeq\G$.
\end{restatable*}
\secref{emd-gr} is devoted to the recovery of an embedded metric graph
$\G\subset\R^d$ from a Hausdorff close, Euclidean sample. We define and study
the relevant properties of the path-based metric $d^\eps$. Finally,
\thmref{h-hom} shows the homotopy equivalence between $\G$ and the
Vietoris--Rips complex $\Ri^\eps_\beta(S)$ of the sample under the metric
$d^\eps$. 
\begin{restatable*}[Homotopy Equivalence under Hausdorff Distance]{theorem}
{hhom}\label{thm:h-hom} Let $\G\subset\R^d$ be an embedded metric graph. Let
$S\subset\R^d$ and $0<\eps<\beta$ be such that
		\[
			4d_{H}(\G,S)<\eps<8\delta(\G)\alpha+2(\delta(\G)+1)\eps\leq\beta
			<\frac{2\rho(\G)}{3\delta(\G)},
		\]
		where $\alpha=(9\delta(\G)+8)\eps$. Then, $\mod{\Ri^\eps_\beta(S)}\simeq\G$.
\end{restatable*}
	
\section{Preliminaries}\label{sec:prelim} In this section, we present
definitions and notations that we use throughout the paper. The standard results
from algebraic topology are stated here without a proof; details can be found
any standard textbook on the subject, e.g., \cite{MUNK,spanier1994algebraic}. 

\subsection{Metric Spaces}
Let $(X,d_X)$ be a metric space. When it is clear from the context, we omit the
metric $d_X$ from the notation, and denote the metric space just by $X$. For any
point $c\in X$ and radius $r\geq0$, the (open) metric ball in $(X,d_X)$ is
denoted by $\B_X(c,r)$. 
\begin{definition}[Diameter]\label{def:diam}
	The diameter, denoted $\diam[X]{Y}$, of a subset
	$Y\subset X$ is defined by the supremum of the pairwise distances in $Y$.
	\[
		\diam[X]{Y}\eqdef\sup_{y_1,y_2\in Y}d_X(y_1,y_2).
	\]
\end{definition}
When $Y$ is compact, its diameter is finite.

A correspondence $\C$ between two (non-empty) metric spaces $(X,d_X)$ and
$(Y,d_Y)$ is defined to be a subset of $X\times Y$ such that
\begin{enumerate}[(a)]
	\item for any $x\in X$, there exists $y\in Y$ such that $(x,y)\in\C$, and
	\item for any $y\in Y$, there exists $x\in X$ such that $(x,y)\in\C$.
\end{enumerate}
We denote the set of all correspondences between $X,Y$ by $\C(X,Y)$. Note that
the definition of correspondence does not depend on the metric space structure
on the sets, however we can define the distortion of a correspondence using the
metrics on them.  For a correspondence $\C\in\C(X,Y)$, its \emph{distortion} is
defined as:
\[
\mathrm{dist}(\C)\stackrel{\text{def}}{=}\sup_{(x_1,y_1),(x_2,y_2)\in\C}
\mod{d_X(x_1,x_2)-d_Y(y_1,y_2)}.
\]
For $\eps>0$, a correspondence is called an \emph{$\eps$-correspondence} if its
distortion is less than $\eps$. 
\begin{definition}[Gromov-Hausdorff Distance]\label{def:gh} Let $(X,d_X)$ and
$(Y,d_Y)$ be two compact metric spaces. The \emph{Gromov-Hausdorff} distance
between $X$ and $Y$, denoted by $d_{GH}(X,Y)$, is defined as:
\[
	d_{GH}(X,Y)\stackrel{\text{def}}{=}
	\frac{1}{2}\left[\inf_{\C\in\C(X,Y)}\mathrm{dist}(\C)\right].
\]
\end{definition}

\subsection{Simplicial Complexes}
An \emph{abstract simplicial complex}~$\K$ is a collection of finite sets such
that if  $\sigma\in\K$, then so are all its non-empty subsets. In general,
elements of $\K$ are called \emph{simplices} of $\K$. The singleton sets in $\K$
are called the \emph{vertices} of $\K$. If a simplex $\sigma\in\K$ has
cardinality $(k+1)$, then it is called a \emph{$k$-simplex} and is denoted by
$\sigma_k$. A $k$--simplex $\sigma_k$ is also written as $[v_0,v_1,\ldots,v_k]$,
where $v_i$'s belong to the vertex set of $\K$. If $\sigma'$ is a (proper)
subset of $\sigma$, then $\sigma'$ is called a (proper) \emph{face} of
$\sigma$, written as $\sigma'\prec\sigma$.

Let $\K_1$ and $\K_2$ be abstract simplicial complexes with vertex sets $V_1$
and $V_2$, respectively. A \emph{vertex map} is a map between the vertex sets.
Let~$\phi \colon V_1 \to V_2$ be a vertex map. We say that $\phi$ induces a
\emph{simplicial map} $\phi:\K_1\to\K_2$ if for all
$\sigma_k=[v_0,v_1,\ldots,v_k]\in K_1$, the image
$$\phi(\sigma_k)\eqdef[\phi(v_0),\phi(v_1),\ldots,\phi(v_k)]$$ is a simplex of
$\K_2$. Two simplicial maps~$\phi,\psi:\K_1\to\K_2$ are called \emph{contiguous}
if for every simplex $\sigma_1\in\K_1$, there exists a simplex $\sigma_2\in\K_2$
such that~$\phi(\sigma_1)\cup\psi(\sigma_1)\prec\sigma_2$. 

For an abstract simplicial complex $\K$ with vertex set $V$, one can define its
\emph{geometric complex} or \emph{underlying topological space}, denoted by
$\mod\K$, as the space of all functions $h:V\to[0,1]$ satisfying the following
two properties:
\begin{enumerate}[(i)]
	\item $\supp(h)\eqdef\{v\in V\mid h(v)\neq0\}$ is a simplex of
	$\K$, and 
	\item $\sum\limits_{v\in\V}h(v)=1$.
\end{enumerate}
For $h\in\mod{\K}$ and vertex $v$ of $\K$, the real number $h(v)$ is called the
$v$--th \emph{barycentric coordinate} of $h$. For a simplex $\sigma$ of $\K$,
its \emph{closed simplex} $\mod{\sigma}$ and \emph{open simplex} $\os{\sigma}$
are subsets of $\mod{\K}$ defined as follows:
\[
	\mod{\sigma}\eqdef\left\{h\in\mod{\K}\mid\supp(h)\subset\sigma\right\},
	\text{and } 
	\os{\sigma}\eqdef\left\{h\in\mod{\K}\mid\supp(h)=\sigma\right\}.
\]
A simplex $\sigma$ of $\K$ is called the \emph{carrier} of a subset
$A\subset\mod{\K}$ if $\sigma$ is the unique smallest simplex such that
$A\subset\mod{\sigma}$. In this work, we use the standard metric topology on
$\mod{\K}$, as defined in \cite{spanier1994algebraic}. A simplicial map
$\phi:\K_1\to\K_2$ induces a continuous (in this topology) map
$\mod\phi:\mod{\K_1}\to\mod{\K_2}$ defined by
\[ \mod\phi(h)(v')\eqdef\sum\limits_{\phi(v)=v'}h(v),\text{ for }v'\in\K_2. \]
From the above definition, it follows that
$\mod{\phi}(h)\in\mod{h(\sigma)}$ whenever $h\in\os{\sigma}$.

A simplicial complex $\K$ is called a \emph{pure $k$--complex} if every simplex
of $\K$ is a face of a $k$--simplex. A simplicial complex $\K$ is called a
\emph{flag complex} if $\sigma$ is a simplex of $\K$ whenever every pair of
points in $\sigma$ is a simplex of $\K$.

\subsection{Barycentric Subdivision}
The \emph{barycenter}, denoted $\bc{\sigma_k}$, of a $k$--simplex
$\sigma_k=[v_0,v_1,\ldots,v_{k}]$ of $\K$ is the point of $\os{\sigma_k}$ such
that $\bc{\sigma_k}(v_i)=\frac{1}{k+1}$ for all $0\leq i\leq k$. Using linearity
of simplices, a more convenient way of writing this is: 
$$\bc{\sigma_k}=\sum_{i=0}^k \frac{1}{k+1}v_i.$$ Let $\K$ be a complex. A
\emph{subdivision} of $\K$ is a simplicial complex $\K'$ such that 
\begin{enumerate}
    \item the vertices of $\K'$ are points of $\mod{\K}$,
    \item if $s'$ is a simplex of $\K'$, then there is $s\in\K$ such that
        $s'\subset\mod{s}$, and
    \item the linear map $h:\mod{\K'}\to\mod{\K}$ sending each vertex of $\K'$
        to the corresponding point of $\mod{\K}$ is a homeomorphism.
\end{enumerate}
For a simplicial complex $\K$, its \emph{barycentric subdivision}, denoted by
$\sd{\K}$, is a special subdivision defined as follows. The vertices of
$\sd{\K}$ are the barycenters of the simplices of $\K$. The simplices of
$\sd{\K}$ are (non-empty) finite sets
$[\bc{\sigma_0},\bc{\sigma_1},\ldots,\bc{\sigma_k}]$ such that
$\sigma_{i-1}\prec\sigma_i$ for $1\leq i\leq k$. If $\sd{\K}$ is further
subdivided, we denote the barycentric subdivision of $\sd{\K}$ by $\sd[2]{\K}$,
and so on. With the definition of the barycentric subdivision at our disposal,
we now prove the following fact, which becomes indispensible for the proofs of
\lemref{gh-sur} and \lemref{h-sur}. 
\begin{proposition}[Commuting Diagram]\label{prop:homotopy} Let $\K$ be a pure
	$k$--complex and $\L$ a flag complex. Let
	$f:~\K\map\L$ and $g:\sd{\K}\map\L$ be simplicial maps such that 
	\begin{enumerate}[(a)]
		\item $g(v)=f(v)$ for every vertex $v$ of $\K$,
		\item $f(\sigma)\cup g(\bc{\sigma})$ is a simplex of $\L$ whenever
		$\sigma$ is a simplex of $\K$.
	\end{enumerate}
	Then, the following diagram commutes up to homotopy: 
	\begin{equation*}
		\begin{tikzpicture} [baseline=(current  bounding  box.center)]
			\node (k1) at (-2,-2) {$\mod{\sd{\K}}$};
			\node (k2) at (2,-2) {$\mod{\K}$};
			\node (k3) at (0,0) {$\mod{\L}$};
			\draw[map] (k2) to node[auto,swap] {$h^{-1}$} (k1);
			\draw[map,swap] (k1) to node[auto,swap] {$\mod{g}$} (k3);
			\draw[map,swap] (k2) to node[auto] {$\mod{f}$} (k3);
		\end{tikzpicture}
	\end{equation*}
	where $h$ is a linear homeomorphism sending each vertex of $\sd{\K}$ to the
    corresponding point of $\mod{\K}$.
\end{proposition}
\begin{proof}
	We show that the maps $\mod{g}$ and $\left(\mod{f}\circ h\right)$ are
	homotopic by constructing an explicit homotopy
	$H:~\mod{\sd{\K}}\times[0,1]\map\mod{\L}$ with $H(\cdot,0)=g(\cdot)$ and
	$H(\cdot,1)=\left(\mod{f}\circ h\right)(\cdot)$. The commutativity of the
	diagram then follows, since $h$ is a homeomorphism. 
	
	The complex $\K$ is taken to be $k$--dimensional. Without any loss of
	generality, we can assume that every point of $\sd{\K}$ belongs to a
	$k$--simplex. Take an arbitrary $x\in\mod{\sd{\K}}$. We can write $x$ in its
	barycentric coordinates as $x=\sum_{i=0}^k\zeta_i\bc{\sigma_i}$, where
	$\sigma_k=[a_0,a_1,\ldots,a_k]$ is a $k$--simplex of $\K$ and
	$\sigma_i=[a_0,a_1,\ldots,a_i]$ for $0\leq i\leq k$. Consider a partition 
	\[0=t_0<t_1<\ldots<t_{k}=1\] of $[0,1]$. We first define the homotopy
	$H(\cdot,t)$ at $t=t_i$ for $i=0,1,\ldots,k$, and then show that we can
	\emph{interpolate} $H$ continuously (using the straight-line homotopy) for
	any $t\in[t_{i-1},t_i]$. For any $0\leq i\leq k$, define 
	\begin{equation}\label{eqn:homotopy}
		H(x,t_i)=\sum_{j=0}^i\zeta_j\left(\mod{f}\circ
		h\right)(\bc{\sigma_j})+\sum_{j=i+1}^k\zeta_j\mod{g}(\bc{\sigma_j}).
	\end{equation} 
	From the above definition, we note for $i=0$ that 
	\begin{align*}
		H(x,0)=\zeta_0\left(\mod{f}\circ
		h\right)(\bc{\sigma_0})+\sum_{j=1}^k\zeta_j\mod{g}(\bc{\sigma_j})
		&=\zeta_0\mod{g}(\bc{\sigma_0})+\sum_{j=1}^k\zeta_j\mod{g}(\bc{\sigma_j})
		\\&=\sum_{j=0}^k\zeta_j\mod{g}(\bc{\sigma_j})=\mod{g}(x).	
	\end{align*}
	The second equality is due to condition (a) and fact that $\sigma_0$ is a
	vertex $\K$. Also, for $i=k$ we note that 
	\[
	H(x,1)=\sum_{j=0}^k\zeta_j\left(\mod{f}\circ h\right)(\bc{\sigma_j})
	=\left(\mod{f}\circ h\right)\left(\sum_{j=0}^k\zeta_j\bc{\sigma_j}\right)
	=\left(\mod{f}\circ h\right)(x).
	\] 
	The second equality is due to the fact that $h$ is a linear homeomorphism.
	
We now fix $0\leq i\leq k$, and extend the definition of $H(x,t)$ for any
$t\in[t_{i-1},t_i]$ by using the straight-line joining $H(x,t_{i-1})$ and
$H(x,t_{i})$. Such an extension is justified, we show now both $H(x,t_{i-1})$
and $H(x,t_{i})$ belong to $\mod{\eta}$ for some simplex $\eta$ of $\L$. 

Let us take $\eta$ to be the set $f(\sigma_i) \cup
g([\bc{\sigma_{i}},\bc{\sigma_{i+1}},\ldots,\bc{\sigma_{k}}])$. Since $f,g$ are
simplicial maps and $\L$ is a flag complex, condition (b) implies that $\eta$ is
a simplex of $\L$. From the definition \eqnref{homotopy} and the fact that
$h(\bc{\sigma_{l}})\in\mod{\sigma_i}$ for all $0\leq l\leq i$, we conclude that
both $H(x,t_{i-1})$ and $H(x,t_{i})$ belong to $\mod{\eta}$. So, $H(x,t)$ is
well-defined and continuous for $t\in[t_{i-1},t_i]$. Moreover, the images agree
at the endpoints as noted from \eqnref{homotopy}. Therefore, $H$ defines the
desired homotopy.
\end{proof}

\subsection{Metric Graphs}
We follow \cite{burago_course_2001} to define a metric graph. We always denote a
metric graph by $\G$ in the following sections.
\begin{definition}[Metric Graph]\label{def:metric-graph} A metric space
    $(\G,d_\G)$ is called a \emph{metric graph} if there are a non-empty set
    $V$, an equivalence relation $\sim$ on $V$, and a collection $E$ of
    metric segments\footnote{a metric segment is a metric space isometric to a
    real line segment $[0,a]$ for $a\geq0$.} with their endpoints in $V$ such
    that $\G$ is homeomorphic to the quotient space
	$$\left(\bigsqcup_{e\in E}e\right)/\sim.$$
\end{definition}
The metric segments $e\in E$ are called the \emph{edges} and the equivalence
classes of their endpoints in $V$ are called the \emph{vertices} of $\G$. For of
an edge $e\in E$, its length $L(e)$ is the length of the corresponding metric
segment. Note that the length can be different from the distance of the endpoint
of $e$ in the $d_\G$ metric. Extending the definition of length for a continuous
path $\gamma:[0,1]\to\G$, we define the length $L(\gamma)$ as the sum of the
lengths of the (full or partial) edges that $\gamma$ consists of. A metric graph
$\G$, therefore, is endowed with two layers of information: a metric structure
turning it into a length space \cite[Definition 2.1.6]{burago_course_2001} and a
combinatorial structure $(V,E)$ as an abstract graph. Using the above
definition, we are also allowing $\G$ to have loops or single edge cycles. In
this paper, we only consider path-connected, locally-finite metric graphs
$(\G,d_\G)$. As a consequence, $d_\G$ is a finite metric and the degree of any
vertex of $G$ is finite. 

For a pair of points $a,b\in\G$, there always exists a \emph{shortest path} or
\emph{geodesic path} $\gamma$ in $\G$ joining $a,b$ such that the length
$L(\gamma)=d_\G(a,b)$.
\begin{definition}[Convexity Radius]\label{def:conv-rad} Let $(\G,d_\G)$ be a
metric graph. We define the convexity radius of $\G$, denoted $\rho(\G)$, to be
the largest number $r\geq0$ with the following property: if $a,b\in\G$ with
$d_\G(a,b)\leq2r$, then there is a unique geodesic path in $\G$ joining $a$ and $b$.
\end{definition}
For a path-connected graph $\G$ with finitely many vertices, the convexity
radius $\rho(\G)$ is always positive. If $e_s$ is a shortest edge, then
$\rho(\G)\geq\frac{1}{4}L(e_s)$. The equality holds when $e_s$ forms a self-loop
in $\G$. The convexity radius can be defined for a more general geodesic spaces;
see \cite{hausmann_1995} for example. It can be shown that the convexity radius
of a closed Riemannian manifold is also positive. Note that there are geodesic
spaces with vanishing convexity radius. When it is positive, however, we have
the following remarkable theorem by Hausmann. Although the result holds for more
general spaces, we state it here in the context of metric graphs. 
\begin{theorem}[Hausmann's Theorem\cite{hausmann_1995}]\label{thm:hausmann} Let
	$(\G,d_\G)$ be a metric graph with a positive convexity radius $\rho(\G)$.
	Then, for any $0<\alpha<\rho(\G)$, the geometric complex of $\Ri_\alpha(\G)$
	is homotopy equivalent to $\G$.
\end{theorem}

\section{Abstract Metric Graphs}\label{sec:abs-gr} Let $(\G,d_\G)$ be a
path-connected (abstract) metric graph with a positive convexity radius
$\rho(\G)$. This section is devoted to the study of the Vietoris--Rips complexes
of a metric space $(S,d_S)$ that is close to $\G$ in the Gromov-Hausdorff
distance (see \defref{gh}). Our main result of this section is presented in
\thmref{gh-hom}. We show how to choose a scale $\beta$, depending only on the
Gromov-Hausdorff distance $d_{GH}(S,\G)$ and $\rho(\G)$, such that the geometric
complex of $\Ri_\beta(S)$ is homotopy equivalent to $\G$.

Our technique of the proof involves the use of the barycentric subdivision of
the Vietoris--Rips complexes of $S$. The approach uses the concept of
circumcenter in metric space---especially in $\G$. In \subsecref{circumcenter},
we define the notion of a circumcenter, and observe that \emph{small} subsets of
$\G$ have a unique circumcenter. We present the final homotopy equivalence
result of this section in \subsecref{homeq-gh}.

\subsection{Circumcenters in $\G$}\label{subsec:circumcenter} We begin with the
definition of circumcenter in a metric space. For the definition, we follow
\cite{burago_course_2001}.
\begin{definition}[Circumcenter]\label{def:circum} Let $(X,d_X)$ be a metric
space and $Y\subset X$ a bounded set. A \emph{circumscribed ball} of $Y$ is a
(closed) metric ball of minimal radius among the balls containing $Y$. The
radius of a circumscribed ball  is called the \emph{circumradius} and the center
a \emph{circumcenter} of $Y$, denote by $\cent{Y}$.
\end{definition}
A circumcenter may not always exist for subsets of a metric space. Even if it
exists, it may not be unique for some subsets. When the diameter of a compact
subset $Y$ of $\G$ is  less than $\rho(\G)$, however, its circumcenter exists
uniquely. For an example, consider a pair of antipodal points as a bounded
subset of a circle.
\begin{proposition}[Circumcenters in Metric Graphs]\label{prop:center} Let
$Y\subset\G$ be a (non-empty) compact subset with $\diam{Y}<\rho(\G)$. Then,
the circumcenter of $Y$ exists uniquely. Moreover, the circumradius is
$\frac{1}{2}\diam{Y}$.
\end{proposition}
\begin{proof}
Since $Y$ is compact, there exist $y_1,y_2\in Y$ such that
$\diam{Y}=d_\G(y_1,y_2)$. From $d_\G(y_1,y_2)<\rho(\G)$, it follows that there
is a unique geodesic in $\G$ joining $y_1$ and $y_2$; call it $\gamma$. We claim
that the midpoint $m$ of $\gamma$ is the unique circumcenter of $Y$, and that
the circumradius is $r=\frac{1}{2}d_\G(y_1,y_2)$. 

The proof becomes trivial if there exists an edge $e$ of $\G$ such that
$\gamma\subset e$. In that case, $Y$ is fully contained in $e$. As a result,
$m$ is the unique circumcenter. 
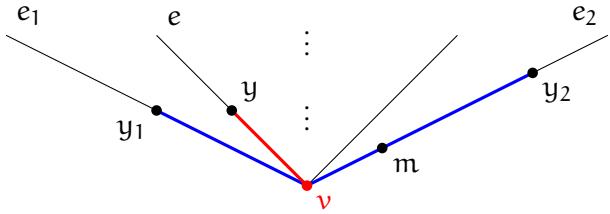
\begin{figure}[thb]
	\centering
	\begin{tikzpicture}[scale=1]
		\draw[] (-4,2) node[anchor=south west] {$e_1$} to (0,0) to (4,2)
			node[anchor=south east] {$e_2$};
		\draw[] (-2,2) node[anchor=south west] {$e$} to (0,0) to (2,2);
		\node at (0,2) {$\vdots$};
		\node at (0,1) {$\vdots$};
		\draw[very thick,blue] (-2,1) to (0,0) to (3,1.5);
		\draw[very thick,red] (-1,1) to (0,0);
		\fill[red] (0,0) node[auto,anchor = north west] {$v$} circle (2pt);
		\fill[black] (1,0.5) node[black,anchor = north west] {$m$} circle (2pt);
		\fill[black] (-2,1) node[black,anchor = north east] {$y_1$} circle (2pt);
		\fill[black] (3,1.5) node[black,anchor = north west] {$y_2$} circle (2pt);
		\fill[black] (-1,1) node[black,anchor = south west] {$y$} circle (2pt);
	\end{tikzpicture}
\caption{For a metric graph $\G$, the star $\st(v)$ of a vertex $v$ is shown.
The black lines indicate some of the incident edges of $v$, including the
designated edges $e_1,e_2$, and $e$. The thick blue line is the geodesic
$\gamma$ joining the points $y_1$ and $y_2$ lying on $e_1$ and $e_2$,
respectively.}
\label{fig:circumcenter}
\end{figure}
We now assume the case when $\gamma$ is contained in two edges $e_1,e_2$ of $\G$
incident to a vertex $v$; the configuration is depicted in
\figref{circumcenter}. Again from $\diam{Y}<\rho(\G)$, we note that
$Y\subset\st(v)$. To see that $Y\subset \overline{\B_{\G}(m,r)}$, take any
$y\in Y$ and consider the following two sub-cases. 

If $y$ lies on either $e_1$ or $e_2$, it has to lie on $\gamma$. Otherwise,
$\diam{Y}>d_\G(y_1,y_2)$---a contradiction. We must have that $y\in\gamma$,
consequently $y\in\overline{\B_{\G}(m,r)}$.

For the other sub-case, let  $y$ lies on an edge $e$ of $\G$ other than
$e_1,e_2$. Without any loss of generality, we can assume that the midpoint $m$
of $\gamma$ lies on $e_2$. This is exactly the case considered in
\figref{circumcenter}. We now note that $d_\G(y,v)\leq d_\G(y_1,v)$. If not,
then
\[ d_\G(y,y_2)=d_\G(y,v)+d_\G(v,y_2)>d_\G(y_1,v)+d_\G(v,y_2)=d_\G(y_1,y_2),\]
again leading to a contradiction. Therefore, the triangle inequality implies
\[
	d_\G(y,m)\leq d_\G(y,v) + d_\G(v,m)\leq d_\G(y_1,v) + d_\G(v,m)
	=\frac{1}{2}d_\G(y_1,y_2)=r.
\]
So, $y\in\overline{\B_{\G}(m,r)}$.

Considering the all the above case, we conclude that
$Y\subset\overline{\B_{\G}(m,r)}$. The uniqueness follows from the observation
that a circumcenter cannot lie outside $\gamma$. Hence, the proof.
\end{proof}

After establishing uniqueness of circumcenters in $\G$, we make another
important observation about the mutual distances of the circumcenters.
\begin{proposition}[Circumcenter Distances]\label{prop:radius} If $Y$ is a
    compact subset of $\G$ with $\diam{Y}<\rho(\G)$, then for any non-empty
    subset $Y'\subset Y$, we have
    $$d_\G\left(\cent{Y'},\cent{Y}\right) \leq\frac{1}{2}\diam{Y}.$$
\end{proposition}
\begin{proof}
In the proof of \propref{center}, we saw that $\cent{Y'}$ is the (geodesic)
midpoint of two farthest points of $Y'$. Letting $r=\frac{1}{2}\diam{Y}$, we
note from \propref{center} that $Y'\subset
Y\subset\overline{\B_\G(\cent{Y},r)}$. Also, the metric ball
$\overline{\B_\G(\cent{Y},r)}$ is a geodesically convex subset of $\G$, due to
the fact that $r<\rho(\G)$. So, we have that
$\cent{Y'}\in\overline{\B_\G(\cent{Y},r)}$. Hence, $d_\G(\cent{Y'},\cent{Y})\leq
r=\frac{1}{2}\diam{Y}$.
\end{proof}

\subsection{Homotopy Equivalence}\label{subsec:homeq-gh} We now assume that
$(S,d_S)$ is a metric space such that the Gromov-Hausdorff distance
$d_{GH}(S,\G)<\frac{\beta}{3}$ for some $\beta>0$. From the definition of
Gromov-Hausdorff distance (\defref{gh}), there must exist a
$\left(\frac{\beta}{3}\right)$--correspondence $\C\in\C(\G,S)$. The
correspondence induces a (possibly non-continuous and non-unique) vertex map
$\phi:\G\map S$ such that $(a,\phi(a))\in\C$ for all $a\in\G$. The vertex map
$\phi$ extends to a simplicial map:
\begin{equation}\label{eqn:phi}
	\Ri_{\frac{2\beta}{3}}(\G)\xrightarrow{\quad\phi\quad}\Ri_\beta(S).
\end{equation}
To see that $\phi$ is a simplicial map, take a $k$--simplex
$\sigma_k=[a_0,a_1,\ldots,a_k]$ in $\Ri_\frac{2\beta}{3}(\G)$. By the
construction of the Vietoris--Rips complex, we have
$d_\G(a_i,a_j)<\frac{2\beta}{3}$ for any $0\leq i,j\leq k$. Since $\C$ is a
$\left(\frac{\beta}{3}\right)$--correspondence and $(a,\phi(a))\in\C$ for all
$a\in\G$, we have
\[
	d_S(\phi(a_i),\phi(a_j))\leq d_\G(a_i,a_j)+\frac{\beta}{3}
	<\frac{2\beta}{3}+\frac{\beta}{3}=\beta.
\]
So, the image $\phi(\sigma_k)=[\phi(a_0),\phi(a_1),\ldots,\phi(a_k)]$ is a
simplex of $\Ri_\beta(S)$.

We show that the simplicial map $\phi$ is a homotopy equivalence. We show in
\lemref{gh-inj} and \lemref{gh-sur} that for any $k\geq0$, the simplicial map
$\phi$ induces an isomorphism between the $k$--th homotopy groups. The homotopy
equivalence of $\phi$ would then follow from Whitehead's theorem
\cite[Theorem 4.5]{HATCH}. 

\begin{remark}\label{rem:basepoint} For the description and computation of
homotopy groups, the consideration of basepoint is deliberately ignored
throughout this paper. This is justified, as the considered scale parameters are
such that all the Vietoris--Rips complexes used here are path-connected. We
prove the claim in \propref{path-connected}.	
\end{remark}

The following two lemmas prove that the simplicial map $\phi$ defined in
\eqnref{phi} induces injective and surjective homomorphisms on the homotopy
groups.
\begin{lemma}[Injectivity]\label{lem:gh-inj} Let $(S,d_S)$ be a metric space
	and $\beta>0$ a number such that 
	\[
		3d_{GH}(\G,S)<\beta<\frac{3\rho(\G)}{4}.
	\]
	For any $k\geq0$, the simplicial map
	$\phi:\Ri_{\frac{2\beta}{3}}(\G)\map\Ri_\beta(S)$ (as defined in
	\eqnref{phi}) induces an injective homomorphism on the $k$--th homotopy
	group.
\end{lemma}
\begin{proof}
Since $d_{GH}(\G,S)<\frac{\beta}{3}$, let us assume that $\C\in\C(\G,S)$ is a
$\left(\frac{\beta}{3}\right)$--correspondence. As a result, we have the
following chain of simplicial maps
\[
\Ri_{\frac{2\beta}{3}}(\G)\xrightarrow{\quad\phi\quad}\Ri_\beta(S)\xrightarrow{\quad\psi\quad}\Ri_{\frac{4\beta}{3}}(\G),
\]
such that $(a,\phi(a))\in\C$ for all $a\in\G$ and $(\psi(b),b)\in\C$ for all
$b\in S$. 
   	
	There is also the natural inclusion
	$\Ri_{\frac{2\beta}{3}}(\G)\xhookrightarrow{\quad
	i\quad}\Ri_{\frac{4\beta}{3}}(\G)$. We first claim that $(\psi\circ\phi)$
	and $i$ are contiguous. To prove the claim, take a $l$--simplex
	$\sigma_k=[a_0,a_1,\ldots,a_l]$ in $\Ri_\frac{2\beta}{3}(\G)$. So,
	$d_\G(a_i,a_j)<\frac{2\beta}{3}$ for all $0\leq i,j\leq l$. We then have
	\begin{align*}
		d_\G((\psi\circ\phi)(a_i),a_j) &=d_\G(\psi(\phi(a_i)),a_j) \\
		&\leq d_S(\phi(a_i),\phi(a_j))+\frac{\beta}{3} \\
		&\leq d_\G(a_i,a_j)+\frac{\beta}{3}+\frac{\beta}{3} \\
		&< \frac{2\beta}{3}+\frac{2\beta}{3}=\frac{4\beta}{3}
	\end{align*}
	This implies that $(\psi\circ\phi)(\sigma_l)\cup i(\sigma_l)$ is a simplex
	of $\Ri_\frac{4\beta}{3}(\G)$. Since $\sigma_l$ is an arbitrary simplex, the
	simplicial maps $(\psi\circ\phi)$ and $i$ are contiguous. Consequently, the
	maps $\mod{\psi\circ\phi}$ and $\mod{i}$ are homotopic.
	
	Since $\frac{4\beta}{3}<\rho(\G)$, \thmref{hausmann} implies that there
	exist homotopy equivalences $T_1,T_2$ such that the following diagram
	commutes (up to homotopy):
	\begin{equation*}
		\begin{tikzpicture} [baseline=(current  bounding  box.center)]
			\node (k1) at (-2,0) {$\mod{\Ri_{\frac{2\beta}{3}}(\G)}$};
			\node (k2) at (2,0) {$\mod{\Ri_{\frac{4\beta}{3}}(\G)}$};
			\node (k3) at (0,-2) {$\G$};
			\draw[rinclusion] (k1) to node[auto] {$\mod{i}$} (k2);
			\draw[map,swap] (k1) to node[auto] {$T_1$} (k3);
			\draw[map,swap] (k2) to node[auto,swap] {$T_2$} (k3);
		\end{tikzpicture}
	\end{equation*}
	So, $\mod{i}$ is also a homotopy equivalence. Hence, the induced
	homomorphism $\mod{i}_*$ on the homotopy groups is an isomorphism. On the
	other hand, we already have $\mod{i}\simeq\mod{\psi\circ\phi}$. Therefore,
	the induced homomorphism $\left(\mod{\psi}_*\circ\mod{\phi}_*\right)$ is
	also an isomorphism, implying that $\mod{\phi}_*$ is an injective
	homomorphism on $\pi_k\left(\Ri_{\frac{2\beta}{3}}(\G)\right)$.
\end{proof}

In order show surjectivity on homotopy groups, we employ the idea of the
simplicial subdivision.
\begin{lemma}[Surjectivity]\label{lem:gh-sur} Let $(S,d_S)$ a metric space and
	$\beta>0$ a number such that 
	\[
		3d_{GH}(\G,S)<\beta<\frac{3\rho(\G)}{4}.
	\]
	Then for any $k\geq0$, the simplicial map
	$\phi:\Ri_{\frac{2\beta}{3}}(\G)\map\Ri_\beta(S)$ (as defined in
	\eqnref{phi}) induces a surjective homomorphism on the $k$--th homotopy
	group.
\end{lemma}
\begin{proof}
	As observed in \propref{path-connected}, both $\Ri_{\frac{2\beta}{3}}(\G)$
	and $\Ri_\beta(S)$ are path-connected. So, the result holds for $k=0$. 
	
For $k\geq1$, let us take an abstract simplicial complex $\K$ such that
$\mod{\K}$ is a triangulation of the $k$-dimensional sphere $\S^k$. Note that
$\K$ is a pure $k$--complex. In order to show surjectivity of $\mod{\phi}_*$ on
$\pi_k\left(\mod{\Ri_{\frac{2\beta}{3}}(\G)}\right)$, we start with a simplicial
map $g:\K\map\Ri_{\beta}(S)$, and argue that there must exist a simplicial map
$\widetilde{g}:\sd{\K}\map\Ri_{\frac{2\beta}{3}}(\G)$ such that the following
diagram commutes up to homotopy:
	\begin{equation}\label{eqn:diag}
		\begin{tikzpicture} [baseline=(current  bounding  box.center)]
			\node (k1) at (-2,0) {$\mod{\Ri_{\frac{2\beta}{3}}(\G)}$};
			\node (k2) at (2,0) {$\mod{\Ri_\beta(S)}$};
			\node (k3) at (2,-2) {$\mod{\K}$};
			\node (k4) at (-2,-2) {$\mod{\sd{\K}}$};
			\draw[map] (k1) to node[auto] {$\mod{\phi}$} (k2);
			\draw[map,swap] (k3) to node[auto] {$\mod{g}$} (k2);
			\draw[map,dashed] (k4) to node[auto] {$\mod{\widetilde{g}}$} (k1);
			\draw[map, <-, dashed] (k4) to node[auto] {$h^{-1}$} (k3);
		\end{tikzpicture}
	\end{equation}
	where the linear homeomorphism $h:\mod{\sd{\K}}\map\mod{\K}$ maps each
	vertex of $\sd{\K}$ to the corresponding point of $\mod{\K}$ as discussed in
	\secref{prelim}.

We note that each vertex of $\sd{\K}$ is the barycenter, $\bc{\sigma}$, of a
simplex $\sigma$ of $\K$. In order to construct the simplicial map
$\widetilde{g}:\sd{\K}\map\Ri_{\frac{2\beta}{3}}(\G)$, we define it on the
vertices of $\sd{\K}$ first, and prove that the vertex map extends to a
simplicial map. Let $\sigma_l=[a_0,a_1,\ldots,a_l]$ be an $l$--simplex of $\K$.
Since $g$ is a simplicial map, then the image
$g(\sigma_l)=[g(a_0),g(a_1),\ldots,g(a_l)]$ is a simplex of $\Ri_\beta(S)$,
hence a subset of $S$ with $\diam[S]{g(\sigma_l)}<\beta$. For each $0\leq i\leq
l$, there exists $a_i'\in\G$ such that $(a_i',g(a_i))\in\C$ for a
$\left(\frac{\beta}{3}\right)$--correspondence $\C\in\C(\G,S)$. For $0\leq i\leq
l$, we denote $\sigma_i'=\{a_0',a_1',\ldots,a_i'\}\subset\G$. Since $\C$ is a
$\left(\frac{\beta}{3}\right)$--correspondence, we note for later that the
diameter of $\sigma_i'$ is (strictly) less than $2\rho(\G)$. 
	\begin{equation}\label{eqn:sigma}
		\diam{\sigma_i'}\leq\diam[S]{g(\sigma_i)}+\frac{\beta}{3}<
		\beta+\frac{\beta}{3}=\frac{4\beta}{3}<\rho(\G).
	\end{equation}
	We then define the vertex map
	\[
		\widetilde{g}(\bc{\sigma_l})\eqdef\cent{\sigma_l'},
	\]
	where $\cent{\sigma_l'}$ is the circumcenter (see \defref{circum}) of
	$\sigma_l'$. Due to the diameter bound in \eqnref{sigma}, \propref{center}
	implies that the circumcenter exists uniquely. To see that $\widetilde{g}$
	extends to a simplicial map, we consider a typical $l$--simplex
	$\tau_l=[\bc{\sigma_0},\bc{\sigma_1},\ldots,\bc{\sigma_l}]$, of $\sd{\K}$,
	where $\sigma_{i-1}\prec\sigma_i\prec\sigma_l$ for $1\leq i\leq l$. Now,
	\begin{align*}
	\diam{\widetilde{g}(\tau_l)}
	&=\diam{[\cent{\sigma_0'},\cent{\sigma_1'},\ldots,\cent{\sigma_l'}]} \\
	&=\max_{0\leq i<j\leq l}\{d_\G(\cent{\sigma_i'},\cent{\sigma_j'})\} \\
	&\leq\max_{0\leq j\leq l}\left\{ \left(\frac{1}{2}\right)\diam{\sigma_j'}
	\right\}, \\
	&\quad\quad\quad\text{ by \propref{radius} as }\diam{\sigma_j'}<\rho(\G) \\
	&=\frac{1}{2}\diam{\sigma_l'} \\
	&<\frac{2\beta}{3},\text{ from \eqnref{sigma}}.
	\end{align*}
	So, $\widetilde{g}(\tau_l)$ is a simplex of $\Ri_{\frac{2\beta}{3}}(\G)$.
	This implies that $\widetilde{g}$ is a simplicial map. 
	
	We invoke \propref{homotopy} to show that the diagram commutes up to
	homotopy. We need to argue that the simplicial maps $g$ and
	$\phi\circ\widetilde{g}$ satisfy the conditions of \propref{homotopy}:
	\begin{enumerate}[(a)]
		\item For any vertex $v\in\K$, 
		\[
			(\phi\circ\widetilde{g})(v)=\phi(\cent{g(v)'})=\phi(g(v)')=g(v).
		\]

		\item For any simplex $\sigma=[a_0,a_1,\ldots,a_k]$ of $\K$, we have
		for $0\leq i\leq k$
		\begin{align*}
			d_S(g(a_i),(\phi\circ\widetilde{g})(\bc{\sigma}))	
			&=d_S(g(a_i),\phi(\cent{\sigma'})) \\
			&\leq d_\G\left(a_i',\cent{\sigma'}\right)+\frac{\beta}{3},
			\text{ since }(a_i',g(a_i))\in\C \\
			&\leq\frac{1}{2}\diam[\G]{\sigma'}+\frac{\beta}{3},
			\text{ by \propref{radius} as }\cent{a_i'}=a_i' \\
			&<\frac{2\beta}{3}+\frac{\beta}{3},
			\text{ from \eqnref{sigma}} \\
			&=\beta.
		\end{align*}
		So, $g(\sigma)\cup(\phi\circ\widetilde{g})(\bc{\sigma})$ is a simplex of
		$\Ri_\beta(S)$.
	\end{enumerate}
	Therefore, \propref{homotopy} implies that the diagram commutes. Since
	$\mod{\K}=\S^k$ and $g$ is arbitrary, we conclude that $\mod{\phi}$ induces
	a surjective homomorphism.
\end{proof}

\ghhom

\begin{proof}
By \lemref{gh-inj} and \lemref{gh-sur}, for any $k\geq0$
\[
	\mod{\phi}_*:\pi_k\left(\mod{\Ri_\frac{2\beta}{3}(\G)}\right)
	\map\pi_k\left(\mod{\Ri_\beta(S)}\right).	
\]
is an isomorphism, where the simplicial map is defined in \eqnref{phi}. By the
Whitehead's theorem, we have that $\mod{\phi}$ is a homotopy equivalence. On the
other hand, since $\frac{2\beta}{3}<\rho(\G)$, \thmref{hausmann} implies that
$\mod{\Ri_\frac{2\beta}{3}(\G)}$ is homotopy equivalent to $\G$. Therefore, we
conclude that $\mod{\Ri_\beta(S)}\simeq\G$.
\end{proof}

\section{Embedded Metric Graphs}\label{sec:emd-gr} This section considers the
Vietoris--Rips complexes of a Euclidean subset $S$ near (in the Hausdorff
distance) to an embedded metric graph $\G$. We first define an embedded metric
graph. For a (continuous) path $\gamma:[0,1]\to\R^d$, we denote by
$L_{\R^d}(\gamma)$ its usual Euclidean length\footnote{For details on the
definition of this length, see \cite[Definition 2.3.1]{burago_course_2001} for
example.}. The path $\gamma$ is called \emph{rectifiable} if $L_{\R^d}(\gamma)$
is finite. Using this length structure, a Euclidean subset $X\subset\R^d$ can be
endowed with yet another metric $d_l$, sometimes called the \emph{length
metric}, defined by
\[
	d_{l}(a,b)\eqdef\inf_{\substack{\gamma:[0,1]\to X \\
	\gamma(0)=a,\gamma(1)=b}}L_{\R^d}(\gamma),
\]
where the infimum is taken over all continuous path $\gamma$ (contained in $X$)
joining $a$ and $b$.
\begin{definition}[Embedded Metric Graph]
	A subset $\G\subset\R^d$ is called an \emph{embedded metric graph} if the
	length metric turns $(\G,d_l)$ into a metric graph (\defref{metric-graph}). 
\end{definition}
If $\G$ is an embedded metric graph, then we denote its length metric by $d_\G$
to retain uniformity with the rest of the paper. We always assume that $\G$ is
path-connected and it has finitely many vertices. We remark that the topologies
on $\G$ induced by the standard Euclidean norm $\norm{\cdot}$ and $d_\G$ are not
always the same. However, the two metrics are equivalent when the distortion of
embedding of $\G$ is finite.
\begin{definition}[Distortion of Embedding]\label{def:delta} The
\emph{distortion of embedding} of a metric graph $\G\subset\R^d$, denoted
$\delta(\G)$, is defined by 
\[
	\delta(\G)\eqdef\sup_{\substack{a,b\in\G \\ a\neq b}}
	\frac{d_\G(a,b)}{\norm{a-b}}	
\] 
\end{definition}
The distortion $\delta(\G)$ is a number greater than $1$, unless $\G$ is a
straight line segment. On the other extreme, $\delta(\G)$ can become infinity;
take $\G=\{(x,y)\in\R^2\mid x^2=y^3\}$ for example. Throughout this paper, we
always assume that the embedded metric graph $\G$ has a finite distortion of
embedding. Then, for any two points $a,b\in\G$, we have
\[
	\norm{a-b}\leq d_\G(a,b)\leq \delta(\G)\norm{a-b}.	
\]
As a consequence, the two metrics $\norm{\cdot}$ and $d_\G$ on $\G$ are
equivalent.

In this section, we consider a subset $S\subset\R^d$ that is close to an
embedded metric graph $\G$ in the Hausdorff distance. As shown for abstract
metric graphs (\thmref{gh-hom}), we wonder if the (Euclidean) Vietoris--Rips
complex of $S$ can also recover an embedded metric graph $\G$. To our
disappointment, we find that however small the Hausdorff distance $d_H(S,\G)$
be, the Euclidean Vietoris--Rips complex of $S$ is not generally homotopy
equivalent to $\G$. As a remedy, we consider the Vietoris--Rips of $S$ under a
parametric family of metrics, we call it the $\eps$--metric $(d^\eps)$. The idea
is to first consider the Euclidean Vietoris--Rips complex of $S$ for a small
scale $\eps>0$, then the shortest path distance in its $1$--skeleton defines the
$d^\eps$ metric on $S$. The metric $d^\eps$ is computable from the pairwise
Euclidean pairwise distances of points in $S$.

In \subsecref{d-eps}, we formally define this metric and explore some useful
properties thereof. We then define in \subsecref{alpha-circum} a variant of the
concept of a circumcenter in the context of $(S,d^\eps)$. Finally,
\subsecref{homeq-h} presents the main homotopy equivalence result for embedded
metric graphs.

\subsection{Path Metric}\label{subsec:d-eps} A (non-empty) subset $S\subset\R^d$
comes equipped with the standard Euclidean metric, given by the Euclidean norm
$\norm{\cdot}$. We define another metric, denoted $d^\eps$, using the pairwise
Euclidean distances of points in $S$. For a positive number $\eps$, we first
introduce the notion of an $\eps$--path.
\begin{definition}[$\eps$--Path]
	Let $S\subset\R^d$ be non-empty and $\eps>0$ a number. For $a,b\in S$, an
	\emph{$\eps$--path} from $a$ to $b$, denoted by $\mathcal{P}^\eps$, is a
	finite sequence $\{y_i\}_{i=0}^{k+1}\subset S$ such that $y_0=a$,
	$y_{k+1}=b$, and $\norm{y_i-y_{i+1}}<\eps$ for all $i=0,1,\ldots,k$.
\end{definition}
The length of $\mathcal{P}^\eps$ is defined by
\[
	L(\mathcal{P}^\eps)\eqdef\sum_{i=0}^k\norm{y_i-y_{i+1}}.	
\]
Now, we are in a position to define the path metric $d^\eps$ on $S$.
\begin{definition}[$d^\eps$--Metric]\label{def:d-eps} Let $S\subset\R^d$ be
    non-empty and $\eps>0$ a number. The $\eps$--metric, denoted $d^\eps$,
    between any $a,b\in S$ is defined by
	\[
		d^\eps(a,b)\eqdef\inf_{\mathcal{P}^\eps}L(\mathcal{P}^\eps),	
	\]
	where the infimum is taken over all $\eps$--paths $\mathcal{P}^\eps$ from
	$a$ to $b$. 
\end{definition}
The metric $d^\eps$ is not finite, in general, for all $\eps>0$. When $d^\eps$
is finite, however, $(S,d^\eps)$ is a metric space. In this metric,
we denote the diameter of a subset $Y\subset S$ by $\diam[\eps]{Y}$. For any
scale $\beta>0$, the Vietoris--Rips complex of $(S,d^\eps)$ is denoted by
$\Ri^\eps_\beta(S)$. For any two points $a,b\in S$, we now compare
$d^\eps(a,b)$ to their standard Euclidean distance $\norm{a-b}$. 
\begin{proposition}\label{prop:s-eps} Let $\eps>0$ be a number and
$S\subset\R^d$ non-empty. For any pair of points $a,b\in S$, we have
\[\norm{a-b}\leq d^\eps(a,b),\]
provided $d^\eps(a,b)$ is finite.
\end{proposition}
\begin{proof}
This follows immediately from \defref{d-eps} and the triangle inequality.
\end{proof}

We now prove in following propositions some key metric properties of
$(S,d^\eps)$, when the subset $S\subset\R^d$ is in a close Hausdorff--proximity
to an embedded metric graph $\G$. As the following proposition shows, a geodesic
of $\G$ can be approximated by an $\eps$--path, with a reasonable bound on its
length. A particular bound has already been found in
\cite{fasy2018reconstruction}. We provide here a more general
result---additionally demonstrating that such as approximation can be made as
accurate as needed by adjusting the Hausdorff distance between $\G$ and $S$. 
See Appendix for a proof.
\begin{proposition}[Approximation of Geodesics by
$\eps$--Paths]\label{prop:geod-appx} 
	Let $\G\subset\R^d$ be an embedded metric graph and $S\subset\R^d$ such that
	$d_H(\G,S)<\frac{1}{2}\xi\eps$ for some $\xi\in(0,1)$ and $\eps>0$. For any
	$a,b\in\G$ and corresponding $a',b'\in S$ with
	$\norm{a-a'},\norm{b-b'}<\frac{1}{2}\xi\eps$, there exists an $\eps$--path
	$\mathcal{P}^\eps$ from $a'$ to $b'$ such that 
	\[
		(1-\xi)L(\mathcal{P}^\eps)\leq d_\G(a,b)+\xi\eps.	
	\]
In particular, $L(\mathcal{P}^\eps)\leq2d_\G(a,b)+\eps$ when
$\xi\leq\frac{1}{2}$.
\end{proposition}

The approximation of geodesics of $\G$ by $\eps$--paths facilitates the
comparison of the $d^\eps$ metric and the geodesic metric $d_\G$.
\begin{proposition}[Comparing Distances]\label{prop:distances} Let
$\G\subset\R^d$ be an embedded metric  graph, and let $S\subset\R^d$ and
$\eps>0$ be such that $d_H(\G,S)<\frac{1}{4}\eps$. For any $a,b\in\G$ and
corresponding $a',b'\in S$ with $\norm{a-a'},\norm{b-b'}<\frac{1}{4}\eps$, we
have
	\begin{equation}\label{eqn:distances}
		\frac{1}{\delta(\G)}d_\G(a,b)-\frac{\eps}{2}\leq d^\eps(a',b')
		\leq2d_\G(a,b)+\eps.
	\end{equation}
\end{proposition}
\begin{proof}
(a) We have
\begin{align*}
	d_\G(a,b) &\leq\delta(\G)\norm{a-b},
	\text{ from definition of the distortion of embedding} \\
	&\leq\delta(\G)\left(\norm{a-a'}+\norm{a'-b'}+\norm{b'-b}\right) \\
	&\leq\delta(\G)\left(\norm{a'-b'}+\frac{\eps}{2}\right) \\
	&\leq\delta(\G)\left(d^\eps(a',b') + \frac{\eps}{2}\right),
	\text{ by \propref{s-eps}}.
\end{align*}
So, the first inequality holds.

(b) For the second inequality, we apply \propref{geod-appx} on the pairs of
points $a,b\in\G$ and $a',b'\in S$, to get an $\eps$--path $\mathcal{P}^\eps$
from $a'$ to $b'$ with $L(\mathcal{P}^\eps)\leq2d_\G(a',b')+\eps$. By
\defref{d-eps}, on the other hand, $d^\eps(a',b')\leq L(\mathcal{P}^\eps)$.
Together they imply 
\[
	d^\eps(a',b')\leq2d_\G(a,b)+\eps.
\]
\end{proof}

We conclude our discussion on the $d^\eps$ metric with the observation that an
$\eps$--path has to lie close to the geodesics of $\G$---provided its length is
bounded in terms of the convexity radius and distortion of embedding of $\G$. We
make this idea concise in the following proposition. A proof is presented in
Appendix.
\begin{proposition}[Approximation of $\eps$--Paths by Geodesics]
	\label{prop:path-appx} 
Let $\G\subset\R^d$ be an embedded metric graph and $S\subset\R^d$ be such that
$4d_H(\G,S)<\eps<~\frac{2\rho(\G)}{\delta(\G)}$ for some $\eps>0$. Let
$\mathcal{P}^\eps$ be an $\eps$--path from $a\in S$ to $b\in S$ with
$\delta(\G)\left(L(\mathcal{P}^\eps)+\frac{\eps}{2}\right)<~\rho(\G)$, and
$a',b'\in\G$ are such that $\norm{a-a'},\norm{b-b'}<\frac{\eps}{4}$. If
$x'\in\G$ is an arbitrary point on the geodesic joining $a'$ and $b'$, then
there exists $y\in\mathcal{P}^\eps$ such that 
\[d^\eps(x,y)<\frac{1}{2}(3\delta(\G)+2)\eps~\forall
x\in S\text{ with }\norm{x-x'}<\frac{\eps}{4}.\]
\end{proposition}

\subsection{$\alpha$--Circumcenter}\label{subsec:alpha-circum}  As observed in
\subsecref{circumcenter}, \emph{small} subsets of $\G$ have a unique
circumcenter. The same can not always be ascertained for \emph{small} subsets of
the metric space $(S,d^\eps)$. We introduce the notion of an
$\alpha$--circumcenter so that properties similar to \propref{center} and
\propref{radius} still hold for the subsets of $(S,d^\eps)$. We now define it
for a general metric space. The inspiration stems from the property of a
circumcenter as demonstrated in \propref{center}.
\begin{definition}[$\alpha$--Circumcenter]
Let $Y$ be a subset of a metric space $(X,d_X)$ and $\alpha\geq0$ a number. A
point $x\in X$ is called an $\alpha$--circumcenter of $Y$, denoted
$\cent[\alpha]{Y}$, if $Y$ is contained in the closed (metric) ball around $x$
of radius $\frac{1}{2}\diam[X]{Y}+\alpha$.
\end{definition}
When $Y$ contains just a single point, we define \[\cent[\alpha]{\{v\}}\eqdef
v\] for any $\alpha$. We also remark that $\cent[\alpha]{Y}$ of a subset $Y$ may
not always exist in a metric space for a given $\alpha$. However, a uniform
$\alpha$ can always be chosen so that $\cent[\alpha]{Y}$ exists for certain
subsets $Y$ of the metric space $(S,d^\eps)$.
\begin{proposition}[$\alpha$--Circumcenter]\label{prop:center-alpha} Let
$\G\subset\R^d$ be an embedded metric graph. Let $S\subset\R^d$ and $\eps>0$ be
such that $4d_H(\G,S)<\eps<\frac{2\rho(\G)}{\delta(\G)}$. For any compact
$Y\subset S$ with
$\delta(\G)\left(\diam[\eps]{Y}+\frac{\eps}{2}\right)<\rho(\G)$, an
$\alpha$--circumcenter $\cent[\alpha]{Y}$ of $Y$ exists for
$\alpha=(9\delta(\G)+8)\eps$.
\end{proposition}
See Appendix for a proof.

\subsection{Homotopy Equivalence for Embedded Metric
Graphs}\label{subsec:homeq-h} Similar to abstract metric graphs, the proof of
homotopy equivalence result for embedded metric graphs is devised using the
barycentric subdivision. Proceeding in the style of \subsecref{homeq-gh}, we
subdivide the Vietoris--Rips complex of $(S,d^\eps)$, but possibly more than
once. For a scale $\beta>0$, we denote the Vietoris--Rips complex of
$(S,d^\eps)$ by $\Ri^\eps_\beta(S)$.

We now assume that $S\subset\R^d$ and $\eps>0$ such that the Hausdorff distance
$d_{H}(\G,S)<\frac{\eps}{4}$ . There is a (possibly non-continuous and
non-unique) vertex map $\phi:\G\map S$ such that
$\norm{a-\phi(a)}<\frac{\eps}{4}$ for all $a\in\G$. From \propref{distances}, we
note that for any $\beta>\eps$ the vertex map $\phi$ extends to a simplicial
map:
\begin{equation}\label{eqn:phi-h}
	\Ri_{\frac{1}{2}(\beta-\eps)}(\G)
	\xrightarrow{\quad\phi\quad}\Ri^\eps_\beta(S).
\end{equation}
We show that the simplicial map $\phi$ is a homotopy equivalence for a suitable
choice of $\beta$. In \lemref{h-inj} and \lemref{h-sur}, we show that for any
$k\geq0$, the simplicial map $\phi$ induces an isomorphism between the $k$--th
homotopy groups.

The following two lemmas prove that the simplicial map $\phi$ defined in
\eqnref{phi-h} induces injective and surjective homomorphisms on the homotopy
groups. The basepoint consideration has been ignored, as we observe from
\propref{distances} that $d^\eps$ is finite, since $\G$ is assumed to be
path-connected. As a result, the geometric complex of $\Ri^\eps_\beta(S)$ is
path-connected, and so are all the other complexes involved in our results.
\begin{lemma}[Injectivity]\label{lem:h-inj} Let $\G\subset\R^d$ be an embedded
	metric graph. Let $S\subset\R^d$ and $0<\eps<\beta$ be such that 
	\[
		4d_{H}(\G,S)<\eps<\beta\leq\frac{2\rho(G)}{3\delta(\G)}.	
	\] 
	 For any $k\geq0$, the simplicial map
	$\phi:\Ri_{\frac{1}{2}(\beta-\eps)}(\G)\map\Ri^\eps_\beta(S)$ (as defined in
	\eqnref{phi-h}) induces an injective homomorphism on the $k$--th homotopy
	group.
\end{lemma}
The proof of surjectivity follows an argument very similar to \lemref{gh-inj},
and is presented in Appendix. In order show surjectivity on homotopy groups, we
again employ the idea of simplicial subdivision.
\begin{lemma}[Surjectivity]\label{lem:h-sur} Let $\G\subset\R^d$ be an embedded
	metric graph. Let $S\subset\R^d$ and $0<\eps<\beta$ be such that
	\[
		4d_{H}(\G,S)<\eps<8\delta(\G)\alpha+2(\delta(\G)+1)\eps\leq\beta
		<\frac{2\rho(\G)}{3\delta(\G)},
	\]
where $\alpha=(9\delta(\G)+8)\eps$. For any $k\geq0$, the simplicial map
$\phi:~\Ri_{\frac{1}{2}(\beta-\eps)}(\G)\map\Ri^\eps_\beta(S)$ (as defined in
\eqnref{phi-h}) induces a surjective homomorphism on the $k$--th homotopy group.
\end{lemma}
\begin{proof}
	As already remarked, the complexes $\Ri_{\frac{1}{2}(\beta-\eps)}(\G)$ and
	$\Ri^\eps_\beta(S)$ are path-connected. So, the result holds for $k=0$.
	
For $k\geq1$, let us take an abstract simplicial complex $\K$ such that
$\mod{\K}$ is a triangulation of the $k$--dimensional sphere $\S^k$. In order to
show surjectivity of $\mod{\phi}_*$, we start with a simplicial map
$g:\K\to\Ri^\eps_\beta(S)$, and argue that there is a natural number $N$ and a
simplicial map $\widetilde{g}:\sd[N]{\K}\map\Ri_{\frac{1}{2}(\beta-\eps)}(\G)$
such that the following diagram commutes up to homotopy:
	\begin{equation}\label{eqn:diag-1}
		\begin{tikzpicture} [baseline=(current  bounding  box.center)]
			\node (k1) at (-2,0) {$\mod{\Ri_{\frac{1}{2}(\beta-\eps)}(\G)}$};
			\node (k2) at (2,0) {$\mod{\Ri^\eps_\beta(S)}$};
			\node (k3) at (2,-2) {$\mod{\K}$};
			\node (k4) at (-2,-2) {$\mod{\sd[N]{\K}}$};
			\draw[rinclusion] (k1) to node[auto] {$\mod{\phi}$} (k2);
			\draw[map,swap] (k3) to node[auto] {$\mod{g}$} (k2);
			\draw[map,dashed] (k4) to node[auto] {$\mod{\widetilde{g}}$} (k1);
			\draw[map, <-, dashed] (k4) to node[auto] {$h^{-1}$} (k3);
		\end{tikzpicture}
	\end{equation}
	where the linear homeomorphism $h:\mod{\sd[N]{\K}}\map\mod{\K}$ maps each
	vertex of $\sd[N]{\K}$ to the corresponding point of $\mod{\K}$.

	\paragraph{Step 1}
	We first note the effect of subdividing $\K$ once. We let
	$\beta_1=\frac{\beta}{2}+\alpha$. We note that
	\begin{align*}
	\beta_1 &=\frac{\beta}{2}+\alpha \\
	&\leq\frac{\beta}{2}+\frac{\beta-(2\delta(\G)+2)\eps}{8\delta(\G)},
	\text{ since }\beta\geq8\delta(\G)\alpha+(2\delta(\G)+2)\eps \\
	&\leq\frac{\beta}{2}+\frac{\beta}{8\delta(\G)} \\
	&\leq\frac{\beta}{2}+\frac{\beta}{2},\text{ since }\delta(\G)\geq1 \\
	&=\beta.
	\end{align*}
	So, the inclusion in \eqnref{diag-2} is justified. Now, we show that there
	exists a simplicial map $g_1:\sd{\K}\to\Ri^\eps_{\beta_1}(S)$ such that the
	following diagram commutes up to homotopy:
	\begin{equation}\label{eqn:diag-2}
		\begin{tikzpicture} [baseline=(current  bounding  box.center)]
			\node (k1) at (-2,0) {$\mod{\Ri_{\beta_1}^\eps(S)}$};
			\node (k2) at (2,0) {$\mod{\Ri^\eps_\beta(S)}$};
			\node (k3) at (2,-2) {$\mod{\K}$};
			\node (k4) at (-2,-2) {$\mod{\sd{\K}}$};
			\draw[rinclusion] (k1) to node[auto] {$\mod{i}$} (k2);
			\draw[map,swap] (k3) to node[auto] {$\mod{g}$} (k2);
			\draw[map,dashed] (k4) to node[auto] {$\mod{g_1}$} (k1);
			\draw[map, <-, dashed] (k4) to node[auto] {$h^{-1}$} (k3);
		\end{tikzpicture}
	\end{equation}
	We first note that each vertex of $\sd{\K}$ is the barycenter,
	$\bc{\sigma_k}$, of a $k$--simplex $\sigma_k$ of $\K$. In order to construct
	the simplicial map $\widetilde{g}:\sd{\K}\map\Ri^\eps_{\beta}(\G)$, we
	define it on the vertices $\sd{\K}$ first, and prove that the vertex map
	extends to a simplicial map. 

	Let $\sigma_k=[a_0,a_1,\ldots,a_k]$ be a $k$--simplex of $\K$. Since $g$ is
	a simplicial map, we have that the image
	$g(\sigma_k)=[g(a_0),g(a_1),\ldots,g(a_k)]$ is a subset of $S$ with
	$\diam[\eps]{g(\sigma_k)}<\beta$. We define 
	\[
	g_1(\bc{\sigma_k})\eqdef\cent[\alpha]{g(\sigma_k)}.
	\]
	To see that $g_1$ extends to a simplicial map, consider a typical
	$k$--simplex, $\tau_k=[\bc{\sigma_0},\bc{\sigma_1},\ldots,\bc{\sigma_k}]$,
	of $\sd{\K}$, where $\sigma_i\prec\sigma_{i+1}$ for $0\leq i\leq
	k-1$ and $\sigma_i\in\K$. Now,
	\begin{align*}
	\diam[\eps]{g_1(\tau_k)}
	&=\diam[\eps]{[\cent[\alpha]{g(\sigma_0)},\cent[\alpha]{g(\sigma_1)},
	\ldots,\cent[\alpha]{g(\sigma_k)}]}\\
	&=\max_{0\leq i<j\leq k}\{d^\eps(\cent[\alpha]{g(\sigma_i)},
	\cent[\alpha]{g(\sigma_j)})\} \\
	&\leq\max_{0\leq j\leq k}\left\{\frac{\diam[\eps]{g(\sigma_j)}}{2}
	+\alpha\right\},\text{ by \propref{center-alpha}}\\
	&\leq\frac{\diam[\eps]{g(\sigma_k)}}{2}+\alpha \\
	&<\frac{\beta}{2}+\alpha. \\
	&=\beta_1
	\end{align*}
	So, $g_1(\tau_k)$ is a simplex of $\Ri^\eps_{\beta_1}(S)$. This implies that
	$g_1$ is a simplicial map. 
	
	We invoke \propref{homotopy} to show that Diagram \eqnref{diag-2} commutes
	up to homotopy. We need to argue that the simplicial maps $g$ and $(i\circ
	g_1)$ satisfy the conditions of \propref{homotopy}:
	\begin{enumerate}[(a)]
		\item For any vertex $v\in\K$, 
		\[
			(i\circ g_1)(v)=i(\cent[\alpha]{g(v)})=i(g(v))=g(v).
		\]

		\item For any simplex $\sigma=[a_0,a_1,\ldots,a_k]$ of $\K$, we have
		for $0\leq i\leq k$
		\begin{align*}
			d_\eps(g(a_i),(i\circ g_1)(\bc{\sigma}))	
			&=d_\eps(g(a_i),\cent[\alpha]{g(\sigma)}) \\
			&=\frac{\diam[\eps]{g(\sigma)}}{2}+\alpha,
			\text{ from \propref{center-alpha}} \\
			&<\frac{\beta}{2}+\alpha=\beta_1\leq\beta.
		\end{align*}
		So, $g(\sigma)\cup(i\circ g_1)(\bc{\sigma})$ is a simplex of
		$\Ri^\eps_\beta(S)$.
	\end{enumerate}
	Therefore, \propref{homotopy} implies that the diagram commutes.

	\paragraph{Step 2}
	Choose a natural number $N$ such that $N\geq2+\log_2{\delta(\G)}$. After $N$
	subdivisions, we then have the following diagram:
	\begin{equation}\label{eqn:step}
		\begin{tikzpicture} [scale=0.72, baseline=(current  bounding  box.center)]
			\node (k1) at (6,0) {$\mod{\Ri^\eps_{\beta}(S)}$};
			\node (k2) at (6,-2) {$\mod{\K}$};
			\node (k3) at (2,0) {$\mod{\Ri^\eps_{\beta_1}(S)}$};
			\node (k4) at (2,-2) {$\mod{\sd{\K}}$};
			\node (k5) at (-2,0) {$\mod{\Ri^\eps_{\beta_2}(S)}$};
			\node (k6) at (-2,-2) {$\mod{\sd[2]{\K}}$};
			\node (k9) at (-7,0) {$\mod{\Ri^\eps_{\beta_N}(S)}$};
			\node (k10) at (-7,-2) {$\mod{\sd[N]{\K}}$};
			\draw[rinclusion] (k3) to node[auto] {$i$} (k1);
			\draw[map,swap] (k2) to node[auto,swap] {$\mod{g}$} (k1);
			\draw[map] (k2) to node[auto] {$h^{-1}$} (k4);
			\draw[rinclusion] (k5) to node[auto] {$i$} (k3);
			\draw[map,swap] (k4) to node[auto] {$\mod{g_1}$} (k3);
			\draw[map] (k6) to node[auto,swap] {$\mod{g_2}$} (k5);
			\draw[map] (k4) to node[auto] {$h^{-1}$} (k6);
			
			\draw[dashed] (k6) to (-3.8,-2);
			\draw[dashed] (-3.2,-.5) to (-3.2,-1.5);
			
			\draw[rinclusion] (k9) to node[auto] {$i$} (k5);
			\draw[map,swap] (k10) to node[auto,swap] {$\mod{g_N}$} (k9);
			\draw[map] (-4.5,-2) to node[auto] {$h^{-1}$} (k10);		
		\end{tikzpicture}
	\end{equation}
	where $\beta_N=\frac{\beta}{2^N}+\alpha\sum_{i=0}^{N-1}\frac{1}{2^i}$.
Moreover, the Diagram \eqnref{step} commutes, since each smaller rectangle
commutes as shown in step 1. 

\paragraph{Step 3}
We also note that
\begin{align*}
	\beta_N &=\frac{\beta}{2^N}+\alpha\sum_{i=0}^{N-1}\frac{1}{2^i} \\
	&\leq\frac{\beta}{2^N}+2\alpha \\
	&\leq\frac{\beta}{4\delta(\G)}+2\alpha,\text{ since }2^N\geq4\delta(\G)\\
	&\leq\frac{\beta}{4\delta(\G)}+\left(\frac{\beta}{4\delta(\G)}
		-\frac{2(1+\delta(\G))\eps}{4\delta(\G)}\right),
		\text{ since }\beta\geq8\delta(\G)\alpha+2(1+\delta(\G))\eps\\
	&=\frac{1}{2\delta(\G)}(\beta-\eps)
		-\frac{\eps}{2}
\end{align*}
This together with $\propref{distances}$ justify the simplicial map $\psi$ in
\eqnref{diag-3}. Moreover, the inclusion $i$ in \eqnref{step} is contiguous to
$(\phi\circ\psi)$. So, the following diagram commutes up to homotopy.
\begin{equation}\label{eqn:diag-3}
	\begin{tikzpicture} [baseline=(current  bounding  box.center)]
		\node (k1) at (-4,0) {$\mod{\Ri_{\beta_N}^\eps(S)}$};
		\node (k2) at (4,0) {$\mod{\Ri^\eps_{\beta}(S)}$};
		\node (k3) at (4,-2) {$\mod{\K}$};
		\node (k4) at (-4,-2) {$\mod{\sd[N]{\K}}$};
		\node (k5) at (0,0) {$\mod{\Ri_{\frac{1}{2}(\beta-\eps)}(\G)}$};
		\draw[map] (k1) to node[auto] {$\mod{\psi}$} (k5);
		\draw[map] (k5) to node[auto] {$\mod{\phi}$} (k2);
		\draw[map,swap] (k3) to node[auto] {$\mod{g}$} (k2);
		\draw[map,dashed] (k4) to node[auto] {$\mod{g_N}$} (k1);
		\draw[map, <-, dashed] (k4) to node[auto] {$h^{-1}$} (k3);
	\end{tikzpicture}
\end{equation}
Therefore, $\widetilde{g}=(\psi\circ g_N)$ is the desired simplicial map. Since
$\mod{\K}=\S^k$ and $g$ is arbitrary, we conclude that $\phi$ induces a
surjective homomorphism on the $k$--th homotopy group.
\end{proof}

\hhom
\begin{proof}
	By \lemref{h-inj} and \lemref{h-sur}, for any $k\geq0$
	\[
		\mod{\phi}_*:\pi_k(\mod{\Ri_{\frac{1}{2}(\beta-\eps)}(\G)})
		\map\pi_k(\mod{\Ri^\eps_\beta(S)}).	
	\]
is an isomorphism, where the simplicial map is defined in \eqnref{phi-h}. By
Whitehead's theorem, we have that $\mod{\phi}$ is a homotopy equivalence. On the
other hand, since $\frac{\beta-\eps}{2}<\rho(\G)$, \thmref{hausmann} implies
that $\mod{\Ri_{\frac{1}{2}(\beta-\eps)}(\G)}$ is homotopy equivalent to $\G$.
Therefore, we conclude that $\mod{\Ri^\eps_\beta(S)}\simeq\G$.
\end{proof}	

\section{Conclusion}
The current work succeeds in providing guarantees for a homotopy type recovery
of a metric graph from the Vietoris--Rips complexes of a metric space close to
it---both in Gromov--Hausdorff and Hausdorff distance. The study provokes a
number of interesting future research directions. Metric graphs are the
simplest, albeit interesting, class of geodesic spaces one can consider. It is
reasonable to believe that recovery of a more general geodesic space, in the
same vein, must require an explicit knowledge of its curvature bounds. To the
best of the author's knowledge, it is not known, even for a two-dimensional
Riemannian manifold, how to choose a suitable scale for the Vietoris--Rips
complex of a metric space Gromov--Hausdorff close it. Although we provide a
homotopy equivalent recovery of an embedded metric graph, the resulting complex
$\Ri^\eps_\beta(S)$, being a very high-dimensional object without a natural
embedding, does not lend itself well to practical applications. Since $S$ is a
subset $\R^d$, one can consider the shadow (as defined in \cite{Chambers2010})
of the complex as a \emph{reconstruction} of $\G$. As pointed out in
\cite[Proposition 5.3]{Chambers2010}, the shadow of a complex is notorious for
being topologically unfaithful. When the Hausdorff between $S$ and $\G$ is very
small, however, we conjecture to have homotopy equivalent shadow of
$\Ri^\eps_\beta(S)$, hence providing a homotopy equivalent reconstruction of
$\G$ with an embedding in the same ambient. 

\paragraph{Acknowledgments}
The author would like to thank the school of information at the University of
California, Berkeley, where this work was completed, for all its support.

\bibliography{main}
\bibliographystyle{plain}

\clearpage
\appendix 
\section{Appendix}
\begin{proposition}[Path-connectedness]\label{prop:path-connected} Let $(S,d_S)$
	be a metric space and $\beta>0$ a number such that $d_{GH}(S,\G)<\beta$,
	then for any positive $\alpha$, the geometric complex of
	$\Ri_{\alpha+\beta}(S)$ is path-connected.
\end{proposition}
\begin{proof}
	Let $a,b\in S$, then there exist points $a',b'\in\G$ such that
	$(a',a),(b',b)\in\C$, where $\C\in\C(\G,S)$ is a $\beta$--correspondence. Since
	$\G$ is assumed to be path-connected, so is $\Ri_\alpha(\G)$. As a result, there
	exists a sequence $\{x_i'\}_{i=0}^{k+1}\subset\G$ forming a path in
	$\Ri_\alpha(\G)$ joining $a'$ and $b'$. In other words, $x_0'=a'$,
	$x_{k+1}'=b'$, and $d_\G(x_i',x_{i+1}')<\alpha$ for $0\leq i\leq k$. There is
	also a corresponding sequence $\{x_i\}_{i=0}^{k+1}\subset S$ such that $x_0=a$,
	$x_{k+1}=b$, and $(x_i',x_i)\in\C$ for all $i$. We note that
	\[
		d_S(x_i,x_{i+1})\leq d_\G(x_i',x_{i+1}')+\beta<\alpha+\beta.	
	\]
	So, the sequence $\{x_i\}$ produces a path in $\Ri_{\alpha+\beta}(S)$ joining
	$a$ and $b$. We conclude that the geometric complex of $\Ri_{\alpha+\beta}(S)$
	is path-connected.
\end{proof}
	
\begin{proof}[Proof of \propref{geod-appx}]
	Let $\gamma:[0,1]\to\G$ be a shortest path on $\G$ joining $a$ and $b$. We can
	find a partition 
	\[0=t_0<t_1<\ldots<t_k<t_{k+1}=1\] of $[0,1]$ such that 
	\begin{equation}\label{eqn:gamma}
	L\left(\gamma\big\lvert_{[t_i,t_{i+1}]}\right)
	=(1-\xi)\eps\text{ for }i=0,1,\ldots,k-1,
	\end{equation}
	and $L\left(\gamma\big\lvert_{[t_k,t_{k+1}]}\right)\leq(1-\xi)\eps$. 
	
	Note that $k$ could also be $0$ if $L(\gamma)<(1-\xi)\eps$. In order to
	construct a $\eps$--path $\mathcal{P}^\eps=\{y_i\}_{i=0}^{k+1}$, we first set
	$y_0=a'$ and $y_{k+1}=b'$. For $i\in\{1,2,\ldots,k\}$, choose $y_i\in S$ such
	that $\norm{\gamma(t_i)-y_i}<\frac{1}{2}\xi\eps$. Since $\gamma(t_0)=a$ and
	$\gamma(t_{k+1})=b$, we have that
	\[
	\norm{\gamma(t_i)-y_i}<\frac{1}{2}\xi\eps\text{ for all }i=0,1,\ldots,k+1.
	\]
	We first show that $\mathcal{P}^\eps$ is, in fact, an $\eps$--path. For each
	$i=0,1,\ldots,k$, from the triangle inequality we get
	\begin{align*}
	\norm{y_i-y_{i+1}}&\leq\norm{y_i-\gamma(t_i)}+\norm{\gamma(t_i)-\gamma(t_{i+1})}
	+\norm{\gamma(t_{i+1})-y_{i+1}}\\
	&<\frac{1}{2}\xi\eps+\norm{\gamma(t_i)-\gamma(t_{i+1})}+\frac{1}{2}\xi\eps \\
	&\leq d_\G(\gamma(t_i),\gamma(t_{i+1}))+\xi\eps \\
	&= L\left(\gamma\big\lvert_{[t_i,t_{i+1}]}\right) + \xi\eps,\text{ since }
	\gamma\text{ is a geodesic} \\
	&\leq(1-\xi)\eps + \xi\eps,\text{ by \eqnref{gamma}} \\
	&=\eps.
	\end{align*}
	Finally,
	\begin{align*}
		d_\G(a,b) 
		&= \sum_{i=0}^k L\left(\gamma\big\lvert_{[t_i,t_{i+1}]}\right) \\
		&= \sum_{i=0}^{k-1} d_\G(\gamma(t_i),\gamma(t_{i+1})) + 
			d_\G(\gamma(t_k),\gamma(t_{k+1})) \\
		&= \sum_{i=0}^{k-1} (1-\xi)\eps + d_\G(\gamma(t_k),\gamma(t_{k+1}))
			\text{, from \eqnref{gamma}} \\
		&> (1-\xi)\sum_{i=0}^{k-1} \norm{y_i-y_{i+1}} + 
			d_\G(\gamma(t_k),\gamma(t_{k+1})), \\
			&\quad\quad\quad\quad
			\text{ as }\norm{y_i-y_{i+1}}>\eps \\
		&\geq (1-\xi)\sum_{i=0}^{k-1}\norm{y_i-y_{i+1}} + 
			\norm{\gamma(t_k)-\gamma(t_{k+1})} \\
		&> (1-\xi)\sum_{i=0}^{k-1}\norm{y_i-y_{i+1}} + 
			\norm{y_k-y_{k+1}}-\xi\eps,
			\text{ since }\norm{\gamma(t_i)-y_i}<\frac{1}{2}\xi\eps \\
		&\geq(1-\xi)\sum_{i=0}^{k-1}\norm{y_i-y_{i+1}} + 
			(1-\xi)\norm{y_k-y_{k+1}}-\xi\eps,\text{ since }0<\xi<1 \\
		&= (1-\xi)\sum_{i=0}^{k}\norm{y_i-y_{i+1}}-\xi\eps \\
		&= (1-\xi)L(\mathcal{P}^\eps)-\xi\eps.
	\end{align*}
\end{proof}
	
\begin{proof}[Proof of \propref{path-appx}]
	We first note from \propref{distances} that
	\[
		d_\G(a',b')\leq\delta(\G)\left(d^\eps(a,b)+\frac{\eps}{2}\right)
		\leq\delta(\G)\left(L(\mathcal{P}^\eps)+\frac{\eps}{2}\right)
		<\rho(\G).
	\]
	So, there exists a unique shortest path, say $\gamma$, joining $a'$ and $b'$.
	Moreover, there exists a vertex $v$ of $\G$ such that $\gamma\subset\st(v)$.
	We now let $\mathcal{P}^\eps=\{y_i\}_{i=0}^{k+1}\subset S$.
	
	Since $d_H(\G,S)<\frac{\eps}{4}$, there exists a corresponding sequence
	$\{y_i'\}_{i=0}^{k+1}\subset\G$ such that $y_0'=a'$ and $y_{k+1}'=b'$, and
	$\norm{y_i-y_i'}<\frac{\eps}{4}$ for all $0\leq i\leq k$. So, we have
	\[
		d_\G(y_i',y_{i+1}')\leq\delta(\G)\norm{y_i'-y_{i+1}'}
		\leq\delta(\G)\left(\norm{y_i-y_{i+1}}+\frac{\eps}{2}\right)
		\leq\delta(\G)\left(L(\mathcal{P}^\eps)+\frac{\eps}{2}\right)<\rho(\G).
	\] 
	Consequently, for each $0\leq i\leq k$, there exists a unique geodesic in $\G$
	joining $y_i'$ and $y_{i+1}'$; call it $\widetilde{\gamma}_i$. Concatenating
	them, we get a (continuous) path $\widetilde{\gamma}:[0,1]\to\G$ joining $a'$
	and $b'$ and a partition 
	\[0=t_0<t_1<\ldots<t_{k}<t_{k+1}=1\] of $[0,1]$ such
	that $\widetilde{\gamma}([t_i,t_{i+1}])=\widetilde{\gamma}_i$ for $0\leq i\leq
	k$.
	
	We now argue that $\gamma\subset\widetilde{\gamma}$. Let us assume the contrary.
	We note that both $\gamma$ and $\widetilde{\gamma}$ are continuous paths on $\G$
	joining the same endpoints---$a'$ and $b'$. The only way $\widetilde{\gamma}$
	fails to cover $\gamma$ if there exists a vertex $w$ of $\G$ such that
	$w\in\widetilde{\gamma}$ and $w\neq v$, i.e.,
	$L(\widetilde{\gamma})\geq2\rho(\G)$. Without loss of generality, then there
	exists $0\leq m\leq k$ such that $d_\G(y_m',b')\geq\rho(\G)$. As a result,
	\begin{align*}
	L(\mathcal{P}^\eps) & \geq\sum_{i=m}^k\norm{y_i-y_{i+1}}\\
	&\geq\norm{y_m-y_{k+1}}\text{, by the triangle inequality} \\
	&\geq\norm{y_m'-b'}-\frac{\eps}{2} \\
	&\geq\frac{1}{\delta(\G)}d_\G(y_{m}',b')-\frac{\eps}{2} \\
	&\geq\frac{\rho(\G)}{\delta(\G)}-\frac{\eps}{2}.
	\end{align*}
	This is a contradiction. So, we have $\gamma\subset\widetilde{\gamma}$.
	Therefore, $x'$ belongs to $\widetilde{\gamma}$. By the construction of
	$\widetilde{\gamma}$, there exists $0\leq l\leq k$ such that
	$x'\in\widetilde{\gamma}_l$. Without any loss of generality, let's assume that
	$d_\G(y_l',x')\leq\frac{1}{2}d_\G(y_l',y_{l+1}')$. Then by \propref{distances},
	\[
		d_\G(y_l',y_{l+1}')\leq\delta(\G)\left(d^\eps(y_l,y_{l+1})+\frac{\eps}{2}\right)
		<\delta(\G)\left(\eps+\frac{\eps}{2}\right)
		\leq\frac{3}{2}\delta(\G)\eps.
	\]
	We set $y=y_l\in\mathcal{P}^\eps$, and note from \propref{distances}, we get
	\[
		d^\eps(x,y)=d^\eps(x,y_l)\leq2d_\G(x',y_l')+\eps\leq2\times
		\frac{1}{2}d_\G(y_l',y_{l+1}')+\eps
		<\frac{1}{2}\left(3\delta(\G)+2\right)\eps.
	\]
\end{proof}

\begin{proof}[Proof of \propref{center-alpha}]
	Since $Y$ is compact, there exist $a,b\in Y$ such that
	$\diam[\eps]{Y}=d^\eps(a,b)$. Let $\eta>0$ be arbitrary. Then there exists an
	$\eps$--path $\mathcal{P}^\eps$ from $a$ to $b$ such that
	$L(\mathcal{P}^\eps)\leq d^\eps(a,b)+\eta$. Let $\theta\in S$ be a point on
	$\mathcal{P}^\eps$ such that
	\[
		d^\eps(a,\theta)+\eta\geq\frac{1}{2}L(\mathcal{P}^\eps)-\eps,\text{ and }
		d^\eps(\theta,b)+\eta\geq\frac{1}{2}L(\mathcal{P}^\eps)-\eps.
	\]
	Since $L(\mathcal{P}^\eps)\geq d^\eps(a,b)$,
	\begin{equation}\label{eqn:theta}
		d^\eps(a,\theta)+\eta\geq\frac{1}{2}d^\eps(a,b)-\eps,\text{ and }
		d^\eps(\theta,b)+\eta\geq\frac{1}{2}d^\eps(a,b)-\eps.
	\end{equation}
	\begin{figure}[thb]
		\centering
		\begin{tikzpicture}[scale=1]
			\draw[dashed,gray,very thick] (-4.5,0) .. controls (0,2) and (0,-2) 
			.. (4.5,0);
			\draw[dashed,gray,very thick] (0,-4.5) .. controls (2,0) and (-2,0) 
			.. (0,4.5);
			\fill[gray] (0,0) node[anchor=south west] {$w$} circle (2pt);
			\draw[fill=red, red] (-3.7,0.5) node[anchor=south] {$a$}
			circle (1pt) to (-3,0.2)
			circle (1pt) to (-2,1)
			circle (1pt) to (-1,0)
			circle (1pt) to (-0.3,-0.1) node[anchor=north] {$d$}
			circle (1pt) to (1,-0.7) node[anchor=north] {$\theta$}
			circle (1pt) to (2,-0.3)
			circle (1pt) to (3,-0.7)
			circle (1pt) to (3.5,0) node[anchor=south] {$b$}
			circle (1pt);
			\draw[fill=blue,blue] (0.5,-3.7) node[anchor=west] {$a_1$}
			circle (1pt) to (0.2,-3)
			circle (1pt) to (1,-2)
			circle (1pt) to (0,-1)
			circle (1pt) to (-0.15,0.5) node[anchor=west,yshift=4pt] {$d_1$}
			circle (1pt) to (-0.7,1) node[anchor=north] {$\theta_1$}
			circle (1pt) to (-0.3,2)
			circle (1pt) to (-0.7,3)
			circle (1pt) to (0,3.5) node[anchor=south] {$b_1$}
			circle (1pt);
			\draw[fill=black,black] (3.5,0) 
			circle (1pt) to (2.3,-0.8)
			circle (1pt) to (1,-0.5)
			circle (1pt) to (-0.4,0.4)
			circle (1pt) to (-0.2,1) node[anchor=south west] {$d_2$}
			circle (1pt) to (-0.7,2) 
			circle (1pt) to (0,2.5)
			circle (1pt) to (0,3.5)
			circle (1pt);
			\fill[gray] (-3.6,0.32) node[anchor=north] {$a'$} circle (1.5pt);
			\fill[gray] (3.5,-0.4) node[anchor=north west] {$b'$} circle (1.5pt);
			\fill[gray] (0.32,-3.6) node[anchor=east] {$a_1'$} circle (1.5pt);
			\fill[gray] (-0.4,3.5) node[anchor=east] {$b_1'$} circle (1.5pt);
		\end{tikzpicture}
		\caption{The star of the vertex $w$ of the metric graph $\G$ is shown. The
		$\eps$--paths $\mathcal{P}^\eps_1$ and $\mathcal{P}^\eps_2$ are shown in red
		and blue, respectively.}
		\label{fig:case-1}
	\end{figure}
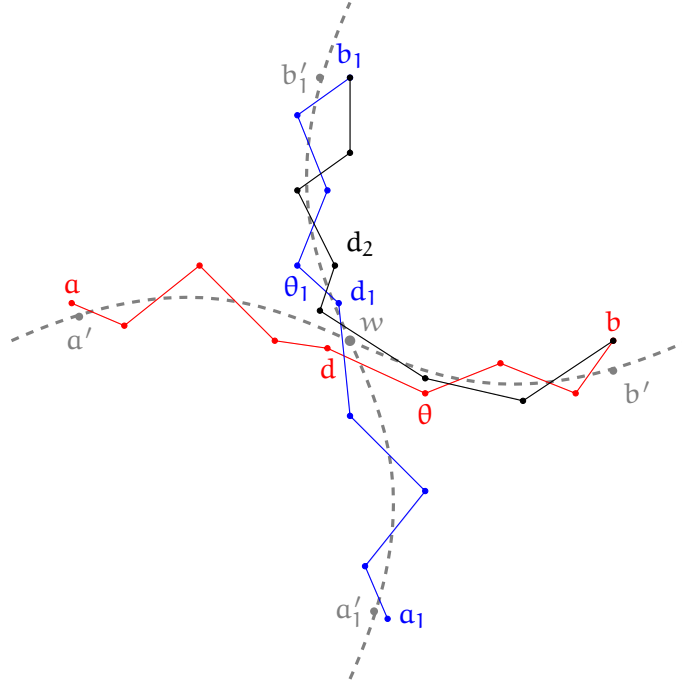
	Similarly for $Y_1$, there exist points $a_1,b_1,\theta_1\in Y_1$ such that
	$\diam[\eps]{Y_1}=d^\eps(a_1,b_1)$, and 
	\begin{equation}\label{eqn:theta1}
		d^\eps(a_1,\theta_1)+\eta\geq\frac{1}{2}d^\eps(a_1,b_1)-\eps,\text{ and }
		d^\eps(\theta_1,b_1)+\eta\geq\frac{1}{2}d^\eps(a_1,b_1)-\eps.
	\end{equation}
	Since $d_H(\G,S)<\frac{\eps}{4}$, there exist corresponding points
	$a',b',a_1',b_1'\in\G$ such that $\norm{a-a'}<\frac{\eps}{4}$,
	$\norm{b-b'}<\frac{\eps}{4}$, $\norm{a_1-a_1'}<\frac{\eps}{4}$, and
	$\norm{b_1-b_1'}<\frac{\eps}{4}$. See \figref{case-1}.
	
	By \propref{distances}, get 
	\[
		d_\G(a',b')\leq\delta(\G)\left(d^\eps(a,b)+\frac{\eps}{2}\right)
		\leq\delta(\G)\left(\diam[\eps]{Y}+\frac{\eps}{2}\right)<\rho(\G).
	\]
	Similarly $d_\G(a_1',b_1')<\rho(\G)$. So, $a',b',a_1',b_1\in\st(w)$ for some
	vertex $w$ of $\G$. For the proof, we consider the most extreme case, when all
	of them lie on two different edges of $\G$ incident to $w$. Let $v\in S$ be such
	that $\norm{v-w}<\frac{\eps}{4}$. Since $w$ lies on the geodesic joining $a,b$,
	\propref{path-appx} implies that there exists $d\in S$ on $\mathcal{P}^\eps$
	such that $d^\eps(v,d)\leq\frac{1}{2}[3\delta(\G)+2]\eps$. Similarly, there
	exists $d_1\in S$ on $\mathcal{P}^\eps_1$ such that
	$d^\eps(v,d_1)\leq\frac{1}{2}[3\delta(\G)+2]\eps$. Without any loss of
	generality, we assume that $\theta$ lies between $d,b$ on $\mathcal{P}^\eps$ and
	$\theta_1$ lies between $d_1,b_1$ on $\mathcal{P}^\eps_1$. Also, there exists an
	$\eps$--path $\mathcal{P}^\eps_2$ from $b$ to $b_1$ such that
	$L(\mathcal{P}^\eps_2)\leq d^\eps(b,b_1)+\eta$. By \propref{path-appx}, there
	exists $d_2\in S$ on $\mathcal{P}^\eps_2$ such that
	$d^\eps(v,d_2)\leq\frac{1}{2}(3\delta(\G)+2)\eps$. From the triangle inequality,
	we get 
	\[d^\eps(d,d_i)\leq(3\delta(\G)+2)\eps\text{ for }i=1,2,\text{ and }
	d^\eps(d_1,d_2)\leq(3\delta(\G)+2)\eps.\]
	Now,
	\begin{align*}
		&d^\eps(b,b_1) +\eta\geq d^\eps(b,d_2) + d^\eps(d_2,b_1) \\
		&\geq [d^\eps(b,d) - d^\eps(d_2,d)] + [d^\eps(b_1,d_1) - d^\eps(d_1,d_2)],
		\text{ by the triangle inequality}\\
		&=d^\eps(b,d) + d^\eps(b_1,d_1) - d^\eps(d,d_2) - d^\eps(d_1,d_2) \\
		&=[d^\eps(b,d)+\eta] + [d^\eps(b_1,d_1)+\eta] 
			- d^\eps(d,d_2) - d^\eps(d_1,d_2) - 2\eta\\
		&\geq [d^\eps(b,\theta) + d^\eps(\theta,d)] 
			+ [d^\eps(b_1,\theta_1) + d^\eps(\theta_1,d_1)]
			- d^\eps(d,d_2) - d^\eps(d_1,d_2) - 2\eta \\
		&\geq \left[\frac{d^\eps(a,b)}{2}-\eps+d^\eps(\theta,d)\right] 
			+ \left[\frac{d^\eps(a_1,b_1)}{2}-\eps+d^\eps(\theta_1,d_1)\right] \\ 
		&\quad\quad- d^\eps(d,d_2) - d^\eps(d_1,d_2) - 4\eta \\
		&\geq \frac{d^\eps(a,b)}{2} + d^\eps(\theta,d) + d^\eps(\theta_1,d_1) 
			- d^\eps(d,d_2) - d^\eps(d_1,d_2) -2\eps - 4\eta,
			\text{ as }d^\eps(a_1,b_1)\geq0.
	\end{align*}
	On the other hand, $d^\eps(a,b)\geq d^\eps(b,b_1)$, because
	$d^\eps(a,b)=\diam[\eps]{Y}$ and $b,b_1\in Y$. So, 
	\[
		d^\eps(\theta,d) + d^\eps(\theta_1,d_1)\leq\frac{d^\eps(a,b)}{2}
		+ d^\eps(d,d_2) + d^\eps(d_1,d_2) + 2\eps + 5\eta.
	\]
	From the triangle inequality, we can then write
	\begin{align*}
	d^\eps(\theta,\theta_1)
	&\leq d^\eps(\theta,d) + d^\eps(d,d_1) + d^\eps(d_1,\theta_1) \\
	&\leq \frac{d^\eps(a,b)}{2}
	+ d^\eps(d,d_2) + d^\eps(d_1,d_2) + 2\eps + 5\eta + d^\eps(d,d_1) \\
	&\leq \frac{d^\eps(a,b)}{2}
	+(d^\eps(d,d_2) + d^\eps(d_1,d_2)+d^\eps(d,d_1)) + 2\eps + 5\eta  \\
	&= \frac{d^\eps(a,b)}{2} + 3\times(3\delta(\G)+2)\eps + 2\eps + 5\eta \\
	&= \frac{\diam[\eps]{Y}}{2} + (9\delta(\G)+8)\eps + 5\eta
	\end{align*}
	Since $\eta$ is arbitrary, we have the result.
\end{proof}

\begin{proof}[Proof of \lemref{h-inj}]
	Since $d_{H}(\G,S)<\frac{\eps}{4}$, from \propref{distances} we have the
	following chain of simplicial maps
	\[
		\Ri_{\frac{1}{2}(\beta-\eps)}(\G)\xrightarrow{\quad\phi\quad}
		\Ri^\eps_\beta(S)\xrightarrow{\quad\psi\quad}
		\Ri_{\delta(\G)\left(\beta+\frac{\eps}{2}\right)}(\G),
	\]
		such that $\norm{a-\phi(a)}<\frac{\eps}{4}$ for all $a\in\G$ and
		$\norm{\psi(b)-b}<\frac{\eps}{4}$ for all $b\in S$. 
		   
		There is also the natural inclusion
		$\Ri_{\frac{1}{2}(\beta-\eps)}(\G)\xhookrightarrow{\quad
		i\quad}\Ri_{\delta(\G)\left(\beta+\frac{\eps}{2}\right)}(\G)$. We first
		claim that $(\psi\circ\phi)$ and $i$ are contiguous. To prove the claim,
		take a $l$--simplex $\sigma_l=[a_0,a_1,\ldots,a_l]$ of
		$\Ri_{\frac{1}{2}(\beta-\eps)}(\G)$. So,
		$d_\G(a_i,a_j)<\frac{1}{2}(\beta-\eps)$ for all $0\leq i,j\leq l$. So, we
		have
		\begin{align*}
			d_\G((\psi\circ\phi)(a_i),a_j) &=d_\G(\psi(\phi(a_i)),a_j) \\
			&\leq2d^\eps(\phi(a_i),\phi(a_j))+\eps,\text{ by \propref{distances}} \\
			&\leq\delta(\G)\left(2d^\eps(a_i,a_j)+\eps+\frac{\eps}{2}\right) \\
			&<\delta(\G)\left((\beta-\eps) + \eps + \frac{\eps}{2}\right) \\
			&=\delta(\G)\left(\beta + \frac{\eps}{2}\right)
		\end{align*}
		This implies that $(\psi\circ\phi)(\sigma_l)\cup i(\sigma_l)$ is a simplex
		of $\Ri_{\delta(\G)\left(\beta+\frac{\eps}{2}\right)}(\G)$. Since
		$\sigma_l$ is an arbitrary simplex, the simplicial maps $(\psi\circ\phi)$
		and $i$ are contiguous. Consequently, the maps $\mod{\psi\circ\phi}$ and
		$\mod{i}$ are homotopic.
		
		Since $\eps<\beta$, we get
		\[
		\delta(\G)\left(\beta+\frac{\eps}{2}\right)
		<\delta(\G)\left(\beta+\frac{\beta}{2}\right)
		=\delta(\G)\left(\frac{3\beta}{2}\right)
		\leq\rho(\G),
		\]
		from \thmref{hausmann} there exist homotopy equivalences $T_1,T_2$ such that
		the following diagram commutes (up to homotopy):
		\begin{equation*}
			\begin{tikzpicture} [baseline=(current  bounding  box.center)]
				\node (k1) at (-3,0) {$\mod{\Ri_{\frac{1}{2}(\beta-\eps)}(\G)}$};
				\node (k2) at (3,0) {$\mod{\Ri_{\delta(\G)\left(\beta+\frac{\eps}{2}\right)}(\G)}$};
				\node (k3) at (0,-2) {$\G$};
				\draw[rinclusion] (k1) to node[auto] {$\mod{i}$} (k2);
				\draw[map,swap] (k1) to node[auto] {$T_1$} (k3);
				\draw[map,swap] (k2) to node[auto,swap] {$T_2$} (k3);
			\end{tikzpicture}
		\end{equation*}
		So, $\mod{i}$ is also a homotopy equivalence. Hence, the induced
		homomorphism $\mod{i}_*$ on the homotopy groups is an isomorphism. On the
		other hand, $\mod{i}\simeq\mod{\psi\circ\phi}$. Therefore, the induced
		homomorphism $(\mod{\psi}_*\circ\mod{\phi}_*)$ is also an isomorphism,
		implying that $\mod{\phi}_*$ is an injective homomorphism.
\end{proof}
\end{document}